\documentclass[12pt]{amsart}
\usepackage[utf8,utf8x]{inputenc}
\usepackage[T1]{fontenc}
\usepackage{a4wide}
\usepackage{eucal}
\usepackage{latexsym}
\usepackage{amssymb}
\usepackage{verbatim}
\usepackage{epic,eepic}
\usepackage{epsfig}
\usepackage[all]{xy}
\usepackage{listings}
\usepackage{rotating}
\usepackage{graphicx}
\usepackage{psfrag}
\usepackage[dvipsnames]{xcolor}
\usepackage{tikz}
\usepackage{multirow}
\usepackage[pdfauthor   = {Mohamed\ Barakat},
            pdftitle    = {Spectral\ Filtrations},
            pdfsubject  = {18G40; 18G10; 18G15; 13D02; 13D07; 16E05; 16E30},
            pdfkeywords = {Grothendieck\ spectral\ sequence; bidualizing\ sequence;\ purity\ filtration},
            bookmarks=true,
            bookmarksopen=true,
            pagebackref=true,
            hyperindex=true,
            colorlinks=true,
            linkcolor=blue,
            citecolor=blue,
            filecolor=blue,
            urlcolor=blue,
            ps2pdf=true
            ]{hyperref}
\usepackage[all]{hypcap}

\input amathmoe.sty

\DeclareMathOperator{\Hom}{Hom}
\DeclareMathOperator{\Ext}{Ext}

\DeclareMathOperator{\Tor}{Tor}
\DeclareMathOperator{\tor}{t}
\DeclareMathOperator{\gr}{gr}
\DeclareMathOperator{\coker}{coker}
\DeclareMathOperator{\img}{im}
\DeclareMathOperator{\Img}{Im}
\DeclareMathOperator{\coim}{coim}

\DeclareMathOperator{\LL}{L}
\DeclareMathOperator{\RR}{R}

\DeclareMathOperator{\Aid}{Aid}
\DeclareMathOperator{\Ann}{Ann}
\DeclareMathOperator{\codim}{codim}
\DeclareMathOperator{\Tot}{Tot}
\DeclareMathOperator{\cp}{cp}

\DeclareMathOperator{\Proj}{Proj}

\newcommand{\tM}{{{\mathtt{M}}}}

\newcommand{\ttF}{{{\mathtt{F}}}}

\newcommand{\id}{{\mathrm{id}}}
\renewcommand{\phi}{{\varphi}}

\newcommand{\Id}{\mathrm{Id}}

\newcommand\cocoa{{\hbox{\rm C\kern-.13em o\kern-.07em C\kern-.13em o\kern-.15em A}}}

\newcommand\colorC{{\color{blue}C}}
\newcommand\colorA{{\color{darkgreen}A}}
\newcommand\colorR{{\color{brown}R}}

\definecolor{darkgreen}{rgb}{0.008,0.617,0.067}
\definecolor{brown}{rgb}{0.6,0.4,0.2}

\newcommand{\lklammer}[1]{      
  \ensuremath{\left\{ \rule{0pt}{#1} \right.}
}
\newcommand{\rklammer}[1]{
  \ensuremath{\left\} \rule{0pt}{#1} \right.}
}

\newcommand{\jumpdir}[2]{\save[]+#2*{#1} \restore}

\newcommand{\kommentar}[1]{}

\begin{document}

\lstset{basicstyle=\tiny,
        frame=single,
        stringstyle=\ttfamily,
        showstringspaces=true}

\author{Mohamed Barakat}
\address{Department of mathematics, University of the Saarland, 66041 Saarbrücken, Germany}
\email{\href{mailto:Mohamed Barakat <barakat@math.uni-sb.de>}{barakat@math.uni-sb.de}}

\title{Spectral Filtrations via Generalized Morphisms}

\begin{abstract}
This paper introduces a reformulation of the classical convergence theorem for spectral sequences of filtered complexes which provides an algorithm to \emph{effectively} compute the induced filtration on the total (co)homology, as soon as the complex is of finite type, its filtration is finite, and the underlying ring is computable. So-called \emph{generalized maps} play a decisive role in simplifying and streamlining all involved algorithms.
\end{abstract}

\maketitle

\tableofcontents

\section{Introduction}

The motivation behind this work was the need for algorithms to explicitly construct several natural filtrations of modules. It is already known that all these filtrations can be described in a unified way using spectral sequences of filtered complexes, which in turn suggests a unified algorithm to construct all of them. Describing this algorithm is the main objective of the present paper.

Since {\sc Verdier} it became more and more apparent that one should be studying complexes of modules rather than single modules. A single module is then represented by one of its resolutions, all quasi-isomorphic to each other. The idea is now very simple:

\smallskip
\begin{center}
\framebox{
\begin{minipage}[c]{0.95\linewidth}
    If there is no direct way to construct a certain natural filtration on a module $M$, it might be simpler to explicitly realize $M$ as one of the (co)homologies $H_n(C)$ of some complex $C$ with some easy constructible (natural) filtration, such that the filtration induced on $H_n(C)$ (by the one on $C$) maps by the explicit isomorphism $H_n(C)\cong M$ onto the looked-for filtration on $M$.
\end{minipage}
 }
\end{center}

\smallskip
In this work it will be shown how to compute the induced filtration on $H_n(C)$ using spectral sequences of filtered complexes, enriched with some extra data. This provides a unified approach for constructing numerous important filtrations of modules and sheaves of modules (cf.~\cite[Chap.~5]{WeiHom} and \cite[Chap.~11]{Rot}). Since we are interested in effective computations we restrict ourself for simplicity to \emph{finite type} complexes carrying \emph{finite} filtrations.

\smallskip
When talking about $D$-modules the ring $D$ is assumed associative with one.

\begin{defn}[Filtered module]
  Let $M$ be a $D$-module.
  \begin{enumerate}
    \item[(a)] A chain of submodules $(F_p M)_{p\in\Z}$ of the module $M$ is called an \textbf{ascending filtration} if $F_{p-1} M \leq F_p M$. The $p$-th \textbf{graded part} is the subfactor module defined by $\gr_p M := F_p M / F_{p-1} M$.
    \item[(d)] A chain of submodules $(F^p M)_{p\in\Z}$ of the module $M$ is called a \textbf{descending filtration} if $F^p M \geq F^{p+1} M$. The $p$-th \textbf{graded part} is the subfactor module defined by $\gr^p M := F^p M / F^{p+1} M$.
  \end{enumerate}
All filtrations of modules will be assumed \textbf{exhaustive} (i.e.\  $\bigcup_p F_p M = M$), \textbf{Hausdorff} (i.e.\  $\bigcap_p F_p M= 0$), and will have \textbf{finite length} $m$ (i.e.\  the difference between the highest and the lowest stable index is at most $m$). Such filtrations are called $m$-step filtrations.
\end{defn}

We start with two examples that will be pursued in Section \ref{appl}:

\begin{enumerate}
  \item[\textbf{(d)}] Let $M$ and $N$ be right $D$-modules and $M^*:=\Hom_D(M,D)$ the dual (left) $D$-module of $M$. The map
    \[
      \phi:
        \left\{
          \begin{array}{ccc}
            N\otimes_D M^* &\to& \Hom_D(M,N) \\
            n \otimes \alpha &\mapsto& (m \mapsto n\alpha(m))
          \end{array}
        \right.
    \]
    is in general neither injective nor surjective. In fact, $\img\phi$ is the last (graded) part of a \textbf{d}escending filtration of $\Hom(M,N)$.
    \begin{equation*}\label{HomMN}
\xymatrix@C=3cm@R=0.62cm{
  &  *=0{\mbox{\small\textbullet}}
       \jumpdir{\Hom(M,N)}{/:a(-90) +1.3cm/}
       \ar@{-}[d]
  \\
  &  *=0{{\mbox{\small\textbullet}}}
       {\ar@{.}[d]}
  \\
  &  *=0{{\mbox{\small\textbullet}}}
       \ar@{-}[d]
  \\
  *=0{\mbox{\small\textbullet}}
  \ar@{-}[d] \jumpdir{N\otimes M^*}{/:a(190) +0.5cm/}
  \ar[r]
  & *=0{{\mbox{\small\textbullet}}}
      {\jumpdir{\rklammer{1.2cm}\coker\phi}{/:a(50)1.65cm/}}
      \ar@{-}[d]
  \\
  *=0{{\mbox{\small\textbullet}}}
  {\ar@{-}[d]}
  \jumpdir{\phi}{/:a(100) +1.6cm/}
  \ar[r]
  {\jumpdir{\phantom{\coim\phi \big\{}}{/:a(160)+1cm/}}
  {\jumpdir{\coim\phi \big\{}{/:a(161)+0.98cm/}}
  & *=0{\mbox{\small\textbullet}}
  {\jumpdir{\big\}\img\phi}{/:a(23)+0.85cm/}}
  \\
  *=0{{\mbox{\small\textbullet}}} {\jumpdir{\ker\phi \big\{}{/:a(160)+0.85cm/}}
}\end{equation*}
  \item[\textbf{(a)}] Dually, let $M$ be a left module, $L$ a right module, and
      \[
        \varepsilon:M \to M^{**}:=\Hom(\Hom(M,D),D)
      \]
      the \textbf{evaluation map}. The composition $\psi$
      \[
       \xymatrix@1{
         L\otimes_D M
         \ar@/_1pc/[rr]_{\psi}
         {\ar[r]^<(.25){\id\otimes\varepsilon}} &
         {L\otimes M^{**} \ar[r]^>(.75)\phi} &
         \Hom_D(M^*,L)
       }
    \]
    is in general neither injective nor surjective. It will turn out that its coimage $\coim\psi$ is the last graded part of an \textbf{a}scending filtration of $L\otimes M$.
    \begin{equation*}\label{MN}
\xymatrix@C=3cm@R=0.62cm{
  &  *=0{\mbox{\small\textbullet}}
       \jumpdir{\Hom(M^*,L)}{/:a(-80) +1.3cm/}
       \ar@{-}[d]
  \\
  *=0{\mbox{\small\textbullet}}
  \ar@{-}[d] \jumpdir{L\otimes M}{/:a(190) +0.5cm/}
  \ar[r]
  & *=0{{\mbox{\small\textbullet}}}
      {\jumpdir{\big\}\coker\psi}{/:a(20)1.07cm/}}
      \ar@{-}[d]
  \\
  *=0{{\mbox{\small\textbullet}}}
  {\ar@{-}[d]}
  \jumpdir{\psi}{/:a(100) +1.6cm/}
  \ar[r]
  {\jumpdir{\phantom{\coim\psi \big\{}}{/:a(160)+1cm/}}
  {\jumpdir{\coim\psi \big\{}{/:a(161)+0.98cm/}}
  & *=0{\mbox{\small\textbullet}}
  {\jumpdir{\big\}\img\psi}{/:a(22)+0.85cm/}}
  \\
  *=0{{\mbox{\small\textbullet}}}
  {\jumpdir{\ker\psi \lklammer{0.98cm}}{/:a(200)+0.9cm/}}
  {\ar@{.}[d]}
  \\
  *=0{{\mbox{\small\textbullet}}}
  \ar@{-}[d]
  \\
  *=0{\mbox{\small\textbullet}}
}\end{equation*}
\end{enumerate}

Example \textbf{(a)} has a geometric interpretation.

\begin{enumerate}
  \item[\textbf{(a')}] Let $D$ be a commutative {\sc Noether}ian ring with $1$. Recall that the {\sc Krull} dimension  $\dim D$ is defined to be the length $d$ of a maximal chain of prime ideals $D>\mathfrak{p}_0> \cdots > \mathfrak{p}_d$. For example, the \textbf{{\sc Krull} dimension} of a field $k$ is zero, $\dim \Z = 1$, and $\dim D[x_1,\ldots,x_n]=\dim D+n$.
  
  The definition of the {\sc Krull} dimension is then extended to nontrivial $D$-modules using
  \[
    \dim M := \dim \frac{D}{\Ann_D(M)}.
  \]
  Define the \textbf{codimension} of a nontrivial module $M$ as
  \[
    \codim M := \dim D - \dim M
  \]
  and set the codimension of the zero module to be $\infty$. If for example $D$ is a (commutative) principal ideal domain which is not a field, then the finitely generated $D$-modules of codimension $1$ are precisely the finitely generated torsion modules.
  \begin{defn}[Purity filtration]
    Let $D$ be a commutative {\sc Noether}ian ring with $1$ and $M$ a $D$-module. Define the submodule $\tor_{-c}M$ as the biggest submodule of $M$ of codimension $\geq c$. The \emph{ascending} filtration
    \[
      \cdots \leq \tor_{-(c+1)}M \leq \tor_{-c}M \leq \cdots \leq \tor_{-1} M  \leq \tor_0 M := M
    \]
    is called the \textbf{purity filtration} of $M$ \cite[Def.~1.1.4]{HL}. The graded part $M_c:=\tor_{-c}/\tor_{-(c+1)}$ is \textbf{pure} of codimension $c$, i.e.\  any nontrivial submodule of $M_c$ has codimension $c$. $\tor_{-1} M$ is nothing but the torsion submodule $\tor(M)$. This suggests calling $\tor_{-c} M$ the \textbf{$c$-th (higher) torsion submodule} of $M$.
  \end{defn}
 
Early references to the purity filtration are {\sc J.-E.~Roos}'s pioneering paper \cite{Roos} where he introduced the \textbf{bidualizing complex}, {\sc M.~Kashiwara}'s master thesis (December~1970) \cite[Theorem 3.2.5]{Kash} on algebraic $D$-modules, and {J.-E.~Björk}'s standard reference \cite[Chap.~2, Thm.~4.15]{Bjo}. All these references address the construction of this filtration from a homological\footnote{{\sc Kashiwara} did not use spectral sequences: ``Instead of using spectral sequences, Sato devised [...] a method using associated cohomology'', \cite[Section~3.2]{Kash}.} point of view, where the assumption of commutativity of the ring $D$ can be dropped.

Under some mild conditions on the \emph{not} necessarily commutative ring $D$ one can characterize the purity filtration in the following way: There exist so-called \textbf{higher evaluation maps} $\varepsilon_c$, generalizing the standard evaluation map, such that the sequence
\[
  0 \xrightarrow{} \tor_{-(c+1)}M \xrightarrow{} \tor_{-c} M \xrightarrow{\varepsilon_c} \Ext^c_D(\Ext^c_D(M,D),D)
\]
is exact (cf.~\cite{AB,QEB}). $\varepsilon_c$ can thus be viewed as a \textbf{natural transformation} between the \textbf{$c$-th torsion functor} $\tor_{-c}$ and the \textbf{$c$-th bidualizing functor} $\Ext^c(\Ext^c(-,D),D)$. In Subsection~\ref{bidual} it will be shown how to use spectral sequences of filtered complexes to construct all the higher evaluation maps $\varepsilon_c$. More generally it is evident that spectral sequences are natural birthplaces for many natural transformations.

Now to see the connection to the previous example (a) set $L=D$ as a right $D$-module. $\psi$ then becomes the evaluation map $\varepsilon$.
\end{enumerate}

There still exists a misunderstanding concerning spectral sequences of filtered complexes and it might be appropriate to address it here. Let $C$ be a filtered complex (cf.~Def.~\ref{filt_complex} and Remark~\ref{comp}). (*) We even assume $C$ of \emph{finite type} and the filtration \emph{finite}. The filtration on $C$ \emph{induces} a filtration on its (co)homologies $H_n(C)$. It is sometimes believed that the spectral sequence $E^r_{pq}$ associated to the filtered complex $C$ cannot be used to determine the induced filtration on $H_n(C)$, but can only be used to determine its graded parts $\gr_p H_n(C)$. One might be easily led to this conclusion since the last page of the spectral sequence consists of precisely these graded parts $E^\infty_{pq}=\gr_p H_{p+q}(C)$, and computing the last page is traditionally regarded as the last step in determining the spectral sequence. It is clear that even the knowledge of the total (co)homology $H_n(C)$ as a whole (along with the knowledge of the graded parts $\gr_p H_n(C)$) is in general \emph{not} enough to determine the filtration. Another reason might be the use of the phrase ``computing a spectral sequence''. Very often this means a successful attempt to figure out the morphisms on some of the pages of the spectral sequence, or even better, working skillfully around determining most or even all of these morphisms and nevertheless deducing enough or even all information about of the last page $E^\infty$. This often makes use of ingenuous arguments only valid in the example or family of examples under consideration. For this reason we add the word \textbf{effective} to the above phrase, and by ``effectively computing the spectral sequence'' we mean \emph{explicitly} determining \emph{all} morphisms on \emph{all} pages of the spectral sequence. Indeed, the definition one finds in standard textbooks like \cite[Section~5.4]{WeiHom} of the spectral sequence associated to a complex of \emph{finite type} carrying a \emph{finite} filtration \emph{is constructive} in the sense that it can be implemented on a computer (see \cite{homalg-package}). The message of this work is the following:

\smallskip
\begin{center}
\framebox{
\begin{minipage}[c]{0.9\linewidth}
    If the spectral sequence of a filtered complex is effectively computable, then, with some extra work, the induced filtration on the total (co)homology is effectively computable as well.
\end{minipage}
 }
\end{center}

\smallskip
By definition, the objects $E^r_{pq}$ of the spectral sequence associated to the filtered complex $C$ are subfactors of the total object $C_{p+q}$ (see Sections~\ref{les} and \ref{filt}). In Section~\ref{genmor} we introduce the notion of a \textbf{generalized embedding} to keep track of this information. The central idea of this work is to use the generalized embeddings $E^\infty_{pq} \to C_{p+q}$ to filter the total (co)homology $H_{p+q}(C)$ --- also a subfactor of $C_{p+q}$. This is the content of Theorem~\ref{mainthm}.

Effectively computing the induced filtration is not a main stream application of spectral sequences. Very often, especially in topology, the total filtered complex is not completely known, or is of \emph{infinite} type, although the (total) (co)homology is known to be of finite type. But from some page on, the objects of the spectral sequence become \emph{intrinsic} and of \emph{finite type}. Pushing the spectral sequence to convergence and determining the isomorphism type of the low degree total (co)homologies is already highly nontrivial. The reader is referred to \cite{RS} and the impressive program {\tt Kenzo} \cite{kenzo}. In its current stage, {\tt Kenzo} is able to compute $A_\infty$-structures on cohomology. The goal here is nevertheless of different nature, namely to effectively compute the induced filtration on the \emph{a priori known} (co)homology. The shape of the spectral sequence starting from the \emph{intrinsic} page will also be used to define new numerical invariants of modules and sheaves of modules (cf.~Subsection~\ref{codegree}).

\smallskip
The approach favored here makes extensive use of \textbf{generalized maps}, a concept motivated in Section~\ref{les}, introduced in Section~\ref{genmor}, and put into action starting from Section~\ref{filt}.

\smallskip
\begin{center}
\framebox{
\begin{minipage}[c]{0.9\linewidth}
Generalized maps can be viewed as a \emph{data structure} that allows \emph{reorganizing} many algorithms in homological algebra as \emph{closed formulas}.
\end{minipage}
 }
\end{center}

\medskip
Although the whole theoretical content of this work can be done over an abstract abelian category, it is sometimes convenient to be able to refer to elements. The discussion in \cite[p.~203]{Har} explains why this can be assumed without loss of generality.

\section{A generality on subobject lattices}

The following situation will be repeatedly encountered in the sequel. Let $C$ be an object in an abelian category, $Z$, $B$, and $A$ subobjects with $B\leq Z$. Then the subobject lattice\footnote{I learned drawing these pictures from Prof.~{\sc Joachim Neubüser}. He made intensive use of subgroup lattices in his courses on finite group theory to visualize arguments and even make proofs.} of $C$ is at most a \textbf{degeneration} of the one in Figure~\ref{ABZ}. 
\begin{figure}[htb]
  \begin{minipage}[c]{0.55\linewidth}
    \centering
    \psfrag{$C$}{$C$}
    \psfrag{$A$}{$A$}
    \psfrag{$B$}{$B$}
    \psfrag{$Z$}{$Z$}
    \psfrag{$A'$}{$A'$}
    \includegraphics[width=0.4\textwidth]{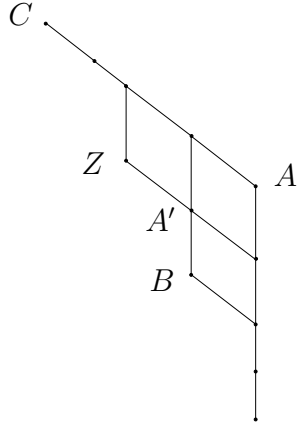}
  \end{minipage}
  \caption{A general lattice with subobjects $B\leq Z$ and $A$}
  \label{ABZ}
\end{figure}

This lattice makes no statement about the ``size'' of $B$ or $Z$ compared to $A$, since, in general, neither $B$ nor $Z$ is in a $\leq$-relation with $A$. The \textbf{second\footnote{Here we follow the numbering in {\sc Emmy Noether}'s fundamental paper \cite{HomSatz}.} isomorphism theorem} can be applied ten times within this lattice, two for each of the five parallelograms. The subobject $A$ leads to the \textbf{intermediate subobject} $A':=(A+B)\cap Z$ sitting between $B$ and $Z$, which in general neither coincides with $Z$ nor with $B$. Hence, a $2$-step filtration $0\leq A \leq C$ leads to a $2$-step filtration $0 \leq A'/B \leq Z/B$. 

Arguing in terms of subobject lattices is a manifestation of the isomorphism theorems, all being immediate corollaries of the homomorphism theorem (cf.~\cite{HomSatz}).

\section{Long exact sequences as spectral sequences}\label{les}

Long exact sequences are in a precise sense a precursor of spectral sequences of filtered complexes. They have the advantage of being a lot easier to comprehend. The core idea around which this work is built can already be illustrated using long exact sequences, which is the aim of this section.

Long exact sequences often occur as the sequence connecting the homologies
\[
  \cdots \leftarrow{} {\color{darkgreen} H_{n-1}(A)} \xleftarrow{{\color{red}\partial_*}}  {\color{brown} H_n(R)}
  \xleftarrow{\nu_*} {\color{blue} H_n(C)} \xleftarrow{\iota_*}
  {\color{darkgreen} H_n(A)}\xleftarrow{{\color{red}\partial_*}} {\color{brown} H_{n+1}(R)} \xleftarrow{} \cdots
\]
of a \emph{short exact} sequence of complexes $0 \xleftarrow{} {\color{brown}R} \xleftarrow{\nu} {\color{blue}C} \xleftarrow{\iota} {\color{darkgreen}A} \xleftarrow{} 0$. If one views $(A,\partial_A)$ as a subcomplex of $(C,\partial)$, then $(R,\partial_R)$ can be identified with the quotient complex $C/A$. Moreover $\partial_A$ is then $\partial_{|A}$ and $\partial_R$ is boundary operator induced by $\partial$ on the quotient $R$. The natural maps $\partial_*$ appearing in the long exact sequence are the so-called connecting homomorphisms and are, like $\partial_A$ and $\partial_R$, induced by the boundary operator $\partial$ of the total complex $C$.

To see in which sense a long exact sequence is a special case of a spectral sequence of a filtered complex we first recall the definition of a filtered complex.

\begin{defn}[Filtered complex]\label{filt_complex}
  We distinguish between chain and cochain complexes:
  \begin{enumerate}
    \item[(a)] A chain of subcomplexes $(F_p C)_{p\in\Z}$ (i.e. $\partial(F_p C_n) \leq F_p C_{n-1}$ for all $n$)
      of the chain complex $(C_\bullet,\partial)$ is called an \textbf{ascending filtration}
      if $F_{p-1} C \leq F_p C$. The $p$-th \textbf{graded part} is the subfactor chain complex
      defined by $\gr_p C := F_p C / F_{p-1} C$.
    \item[(d)] A chain of subcomplexes $(F^p C^n)_{p\in\Z}$ (i.e. $\partial(F^p C^n) \leq F^p C^{n+1}$ for all $n$)
      of the \emph{co}chain complex $(C^\bullet,\partial)$ is called a \textbf{descending filtration}
      if $F^p C \geq F^{p+1} C$. The $p$-th \textbf{graded part} is the subfactor cochain complex
      defined by $\gr^p C := F^p C / F^{p+1} C$.
  \end{enumerate}
Like for modules all filtrations of complexes will be \textbf{exhaustive} (i.e.\  $\bigcup_p F_p C = C$), \textbf{Hausdorff} (i.e.\  $\bigcap_p F_p C= 0$), and will have \textbf{finite length} $m$ (i.e.\  the difference between the highest and the lowest stable index is at most $m$). Such filtrations are called $m$-step filtrations in the sequel.

Convention: For the purpose of this work filtrations on chain complexes are automatically ascending whereas on \emph{co}chain complexes descending.
\end{defn}

\begin{rmrk}
Before continuing with the previous discussion it is important to note that
  \begin{enumerate}
    \item[(a)] The filtration $(F_p C_n)$ of $C_n$ \emph{induces} an ascending filtration
      on the homology $H_n(C)$. Its $p$-th graded part is denoted by $\gr_p H_n(C)$.
    \item[(d)] The filtration $(F^p C^n)$ of $C^n$ \emph{induces} a descending filtration
      on the cohomology $H^n(C)$.  Its $p$-th graded parts is denoted by $\gr^p H^n(C)$.
  \end{enumerate}
  More precisely, $F_p H_n(C)$ is the image of the morphism $H_n(F_p C) \to H_n(C)$.
\end{rmrk}

A short exact sequence of (co)chain complexes $0 \xleftarrow{} {\color{brown}R} \xleftarrow{\nu} {\color{blue}C} \xleftarrow{\iota} {\color{darkgreen}A} \xleftarrow{} 0$ can be viewed as a $2$-step filtration $0\leq A \leq C$ of the complex $C$ with graded parts $A$ and $R$. Following the above convention the filtration is ascending or descending depending on whether $C$ is a chain or cochain complex.

The main idea behind long exact sequences is to relate the homologies of the total chain complex $C$ with the homologies of its graded parts $A$ and $R$. This precisely is also the idea behind spectral sequences of filtered complexes but generalized to $m$-step filtrations, where $m$ may now be larger than $2$. Roughly speaking, the spectral sequence of a filtered complex measures how far the graded part $\gr_p H_n(C)$ of the filtered $n$-th homology $H_n(C)$ of the total filtered complex $C$ is away from simply being the homology $H_n(\gr_p C)$ of the $p$-th graded part of $C$. This would for example happen if the filtration $F_p C$ is induced by its own grading\footnote{In the context of long exact sequences this would mean that the short exact sequence of complexes $0 \xleftarrow{} Q \xleftarrow{\nu} C \xleftarrow{\iota} T \xleftarrow{} 0$ splits.}, i.e.\  $F_p C = \bigoplus_{p' \leq p} \gr_p' C$, since then the homologies of $C$ will simply be the direct sum of the homologies of the graded parts $\gr_p C$. In general, $\gr_p H_n(C)$ will only be a \emph{subfactor} of $H_n(\gr_p C)$.

Long exact sequences do not have a direct generalization to $m$-step filtrations, $m>2$. The language of spectral sequences offers in this respect a better alternative. In order to make the transition to the language of spectral sequences notice that the graded parts $\coker(\iota_*)$ and $\ker(\nu_*)$ of the filtered total homology $H_n(C)$ indicated in the diagram below
\begin{equation}\label{LES}
  \xymatrix@R=0.4cm{
    {\color{darkgreen}H_{n-1}(A)}
    &
    \color{brown}H_n(R)
    {\ar[l]_<(.2){\color{red}\partial_*}}
    &
    {\color{blue}H_n(C)}
    {\ar[l]_<(.18){\nu_*}}
    &
    {\color{darkgreen}H_n(A)}
    {\ar[l]_<(.18){\iota_*}}
    &
    {\color{brown}H_{n+1}(R)}
    {\ar[l]_{\color{red}\partial_*}}
    \\
    *=0{{\color{darkgreen}\mbox{\tiny\textbullet}}}
    {\ar@{-}[d]}
    \\
    *=0{{\mbox{\tiny\textbullet}}}
    {\ar@{-}[d]}
    &
    *=0{{\color{brown}\mbox{\tiny\textbullet}}}
    {\ar@{-}[d]}
    {\color{red}\ar[l]}
    \\
    *=0{{\mbox{\tiny\textbullet}}}
    &
    *=0{{\mbox{\tiny\textbullet}}}
    {\ar@{-}[d]}
    {\ar[l]}
    {\jumpdir{\color{red}\partial_*}{/:a(-11)1.1cm/}}
    &
    *=0{\color{blue}\mbox{\tiny\textbullet}}
    \ar@{-}[d]
    {\ar[l]}
    \\
    &
    *=0{\mbox{{\tiny\textbullet}}}
    &
    *=0{{\mbox{\tiny\textbullet}}}
    \ar@{-}[d]
    {\ar[l]}
    {\jumpdir{\nu_*}{/:a(-13)1.07cm/}}
    {\jumpdir{\iota_*}{/:a(-17)-1.1cm/}}
    {\jumpdir{{\color{brown}\}}\mbox{ \tiny$\coker(\iota_*)$}}{/:a(19)-0.9cm/}}
    {\jumpdir{\mbox{\tiny$\ker(\nu_*)$ }{\color{darkgreen}\big\{}}{/:a(27)0.7cm/}}
    &
    *=0{{\color{darkgreen}\mbox{\tiny\textbullet}}}
    {\ar@{-}[d]}
    {\ar[l]}
    \\
    &
    &
    *=0{\mbox{\tiny\textbullet}}
    &
    *=0{{\mbox{\tiny\textbullet}}}
    {\ar@{-}[d]}
    {\ar[l]}
    &
    *=0{{\color{brown}\mbox{\tiny\textbullet}}}
    {\ar@{-}[d]}
    {\ar[l]}
    \\
    &
    &
    &
    *=0{{\mbox{\tiny\textbullet}}}
    &
    *=0{{\mbox{\tiny\textbullet}}}
    {\ar@{-}[d]}
    {\ar[l]}
    {\jumpdir{\color{red}\partial_*}{/:a(-15)1.1cm/}}
    \\
    &
    &
    &
    &
    *=0{{\mbox{\tiny\textbullet}}}
  }
\end{equation}
both have an alternative description in terms of the connecting homomorphisms:
\begin{equation}\label{iota_nu}
 \coker(\iota_*) \cong {{\color{brown}\ker}({\color{red}\partial_*}) \quad\quad \mbox{and} \quad \quad \ker(\nu_*) \cong \color{darkgreen}\coker}({\color{red}\partial_*}).
\end{equation}
These natural isomorphisms are nothing but the statement of the homomorphism theorem applied to $\iota_*$ and $\nu_*$.

Below we will give the definition of a spectral sequence and in Section \ref{filt} we will recall how to associate a spectral sequence to a filtered complex. But before doing so let us describe in simple words the rough picture, valid for general spectral sequences (even for those not associated to a filtered complex).

A spectral sequence can be viewed as a book with several pages $E^a$, $E^{a+1}$, $E^{a+2}$, $\ldots$ starting at some integer $a$. Each page contains a double array $E^r_{pq}$ of objects, arranged in an array of complexes. The pattern of arranging the objects in such an array of complexes depends only on the integer $a$ and is fixed by a common convention once and for all. The objects on page $r+1$ are the homologies of the complexes on page $r$. It follows that the object $E^r_{pq}$ on page $r$ are \emph{subfactors} of the objects $E^t_{pq}$ on \emph{all} the previous pages $t<r$.

Now we turn to the morphisms of the complexes. From what we have just been saying we know that at least the source and the target of a morphism on page $r+1$ are completely determined by page $r$. This can be regarded as a sort of restriction on the morphism, and indeed, in the case when zero is the only morphism from the given source to the given target, the morphism then becomes uniquely determined. This happens for example whenever either the source or the target vanishes, but may happen of course in other situations ($\Hom_\Z(\Z/2\Z,\Z/3\Z)=0$). So now it is natural to ask whether page $r$ or any of its previous pages impose further restrictions on the morphisms on page $r+1$, apart from determining their sources and targets. The answer is, in general, no. This will become clear as soon as we construct the spectral sequence associated to a $2$-step filtered complex below (or more generally for an $m$-step filtration in Section \ref{filt}) and understand the nature of data on each page.

Summing up: Taking homology only determines the objects of the complexes on page $r+1$, but not their morphisms. Choosing these morphisms not only completes the $(r+1)$-st page, but again determines the objects on the $(r+2)$-nd page. Iterating this process finally defines a spectral sequence.

Typically, in applications of spectral sequences there exists a natural choice of the morphisms on the successive pages. This is illustrated in the following example, where we associate a spectral sequence to a $2$-filtered complex. But first we recall the definition of a spectral sequence.

\begin{defn}[Homological spectral sequence]
  A \textbf{homological spectral sequence} (starting at $r_0$) in an abelian category $\mathcal{A}$ consists of
  \begin{enumerate}
    \item Objects $E^r_{pq} \in \mathcal{A}$, for $p,q,r\in\Z$ and $r \geq r_0\in \Z$; arranged as a sequence (indexed by $r$) of lattices (indexed by $p,q$);
    \item Morphisms $\partial^r_{pq}:E^r_{pq} \to E^r_{p-r,q+r-1}$ with $\partial^r \partial^r=0$, i.e. the sequences of slope $-\frac{r+1}{r}$ in $E^r$ form a chain complex;
    \item Isomorphisms between $E^{r+1}_{pq}$ and the homology $\ker \partial^r_{pq}/\img \partial^r_{p+r,q-r+1}$ of $E^r$ at the spot $(p,q)$.
  \end{enumerate}
  $E^r$ is called the $r$-th \textbf{sheet} (or \textbf{page}, or \textbf{term}) of the spectral sequence.
\end{defn}
Note that $E^{r+1}_{pq}$ is by definition (isomorphic to) a subfactor of $E^r_{pq}$. $p$ is called the \textbf{filtration degree} and $q$ the \textbf{complementary degree}. The sum $n=p+q$ is called the \textbf{total degree}. A morphism with source of total degree $n$, i.e.\  on the $n$-th diagonal, has target of degree $n-1$, i.e. on the $(n-1)$-st diagonal. So the total degree is \emph{decreased} by one.

\begin{figure}[htb]
  \begin{minipage}[c]{1\linewidth}
\[
  \xymatrix{
    *=0{}
     \jumpdir{q}{/:a(90).5cm/}
    &
    E^2_{02}
    \ar@{.}[rd]
    &
    E^2_{12}
    \ar@{.}[rd]
    &
    E^2_{22}
    \\
    &
    E^2_{01}
    \ar@{.}[rd]
    &
    E^2_{11}
    \ar@{.}[rd]
    &
    E^2_{21}
    {\ar[llu]_{{\color{red}\partial}}}
    \\
    &
    E^2_{00}
    &
    E^2_{10}
    &
    E^2_{20}
    {\ar[llu]_{{\color{red}\partial}}}
    \\
    *=0{} \ar[uuu] \ar[rrr]
    &&&
    *=0{}
    \jumpdir{p}{/:a(-100).5cm/}
  }
\]
\end{minipage}
  \caption{$E^2$}
  \label{E_2_with_arrows}
\end{figure}

\begin{defn}[Cohomological spectral sequence]
  A \textbf{cohomological spectral seq\-uence} (starting at $r_0$) in an abelian category $\mathcal{A}$ consists of
  \begin{enumerate}
    \item Objects $E_r^{pq} \in \mathcal{A}$, for $p,q,r\in\Z$ and $r \geq r_0\in \Z$; arranged as a sequence (indexed by $r$) of lattices (indexes by $p,q$);
    \item Morphisms $d_r^{pq}:E_r^{pq} \to E_r^{p+r,q-r+1}$ with $d_r d_r=0$, i.e. the sequences of slope $-\frac{r+1}{r}$ in $E_r$ form a cochain complex;
    \item Isomorphisms between $E_{r+1}^{pq}$ and the cohomology of $E_r$ at the spot $(p,q)$.
  \end{enumerate}
  $E_r$ is called the $r$-th \textbf{sheet} of the spectral sequence.
\end{defn}
Here the total degree $n=p+q$ is \emph{increased} by one. Reflecting a cohomological spectral sequence at the origin $(p,q)=(0,0)$, for example, defines a homological one $E^r_{pq}=E_r^{-p,-q}$, and vice versa. For more details and terminology (\textbf{boundedness}, \textbf{convergence}, \textbf{fiber terms}, \textbf{base terms}, \textbf{edge homomorphisms}, \textbf{collapsing}, \textbf{$E^\infty$ term}, \textbf{regularity}) see \cite[Section~5.2]{WeiHom}.

Part of the data we have in the context of long exact sequences can be put together to construct a spectral sequence with three pages $E^0$, $E^1$, and $E^2$:
\[
\xymatrix{
  E^{0}_{pq}:
  &
  {\color{darkgreen}A_n}
  \ar@{.}[rd]
  &
  {\color{brown}R_{n+1}}
  \\
  &
  {\color{darkgreen}A_{n-1}}
  \ar@{.}[rd]
  &
  {\color{brown}R_n}
  \\
  &
  {\color{darkgreen}A_{n-2}}
  &
  {\color{brown}R_{n-1}}
}
\xymatrix{
&& \\
& \ar@{~>}[r]^{\textrm{add the}}_{\textrm{arrows}} &
}
\xymatrix{
  E^{0}_{pq}:
  &
  {\color{darkgreen}A_n}
  {\ar[d]^{\partial_\colorA}}
  \ar@{.}[rd]
  &
  {\color{brown}R_{n+1}}
  {\ar@{}[d]^{\phantom{\partial_\colorR}}}
  {\ar[d]^{\partial_\colorR}}
  \\
  &
  {\color{darkgreen}A_{n-1}}
  {\ar[d]^{\partial_\colorA}}
  \ar@{.}[rd]
  &
  {\color{brown}R_n}
  {\ar[d]^{\partial_\colorR}}
  \\
  &
  {\color{darkgreen}A_{n-2}}
  &
  {\color{brown}R_{n-1}}
}
\]
\[
\xymatrix{
& \ar@{~>}[dl]^{\textrm{homology}}_{\textrm{take}} \\
\mbox{\phantom{$1$}} &
}
\]
\[
\xymatrix{
  E^{1}_{pq}:
  &
  \ar@{.}[rd]
  {\color{darkgreen}H_n(A)}
  &
  {\color{brown}H_{n+1}(R)}
  \\
  &
  \ar@{.}[rd]
  {\color{darkgreen}H_{n-1}(A)}
  &
  {\color{brown}H_n(R)}
  \\
  &
  {\color{darkgreen}H_{n-2}(A)}
  &
  {\color{brown}H_{n-1}(R)}
}
\xymatrix{
&& \\
& \ar@{~>}[r]^{\textrm{add the}}_{\textrm{arrows}} &
}
\xymatrix{
  E^{1}_{pq}:
  &
  \ar@{.}[rd]
  {\color{darkgreen}H_n(A)}
  &
  {\color{brown}H_{n+1}(R)}
  {\ar[l]_{\color{red}\partial_*}}
  \\
  &
  \ar@{.}[rd]
  {\color{darkgreen}H_{n-1}(A)}
  &
  {\color{brown}H_n(R)}
  {\ar[l]_{\color{red}\partial_*}}
  \\
  &
  {\color{darkgreen}H_{n-2}(A)}
  &
  {\color{brown}H_{n-1}(R)}
  {\ar[l]_{\color{red}\partial_*}}
}
\]
\[
\xymatrix{
& \ar@{~>}[dl]^{\textrm{homology}}_{\textrm{take}} \\
\mbox{\phantom{$1$}} &
}
\]
\[
\xymatrix{
  E^{2}_{pq}:
  &
  \ar@{.}[rd]
  {{\color{darkgreen}\coker}({\color{red}\partial_*})}
  &
  {{\color{brown}\ker}({\color{red}\partial_*})}
  \\
  &
  \ar@{.}[rd]
  {{\color{darkgreen}\coker}({\color{red}\partial_*})}
  &
  {{\color{brown}\ker}({\color{red}\partial_*})}
  \\
  &
  {{\color{darkgreen}\coker}({\color{red}\partial_*})}
  &
  {{\color{brown}\ker}({\color{red}\partial_*})}
}
\xymatrix{
&& \\
& \ar@{~>}[r]^{\textrm{no arrows}}_{\textrm{to add}} &
}
\xymatrix{
  E^{2}_{pq}:
  &
  \ar@{.}[rd]
  {{\color{darkgreen}\coker}({\color{red}\partial_*})}
  &
  {{\color{brown}\ker}({\color{red}\partial_*})}
  \\
  &
  \ar@{.}[rd]
  {{\color{darkgreen}\coker}({\color{red}\partial_*})}
  &
  {{\color{brown}\ker}({\color{red}\partial_*})}
  \\
  &
  {{\color{darkgreen}\coker}({\color{red}\partial_*})}
  &
  {{\color{brown}\ker}({\color{red}\partial_*})}
}
\]
with $p,q\in\Z$, $n=p+q$. Taking the two columns over $p=0$ and $p=1$, for example, is equivalent to setting $F_{-1} C:=0$, $F_0 C:=A$, and $F_1 C:=C$.

Several remarks are in order. First note that all the arrows in the above spectral sequence are induced by $\partial$, the boundary operator of the total complex $C$. Since $\partial$ respects the filtration, i.e. $\partial(F_p C) \leq F_p C$, the induced map $\bar{\partial}:F_p C \to C/F_p C$ vanishes. So respecting the filtration means that $\partial$ cannot carry things up in the filtration. But since $\partial$ does not necessarily respect the grading induced by the filtration it may very well carry things down one or more levels. Now we can interpret the pages: $E^0$ consists of the graded parts $\gr_p C$ with boundary operators $\partial_A$ and $\partial_Q$ chopping off all what $\partial$ carries down in the filtration. $E^1$ describes what $\partial$ carries down exactly one level. This interpretation of the connecting homomorphisms $\partial_*$ puts them on the same conceptual level as $\partial_A$ and $\partial_Q$. Finally, $E^2$ describes what $\partial$ carries exactly two levels down, but since a $2$-step filtration has two levels it should now be clear why $E^2$ does not have arrows.

Second, as we have seen in (\ref{iota_nu}) using the homomorphism theorem, the objects of the last page $E^2$ can be naturally identified with the graded parts $\gr_p H_n(C)$ of the filtered total homology $H_n(C)$. And since the objects on each page are subfactors of the objects on the previous pages one can view the above spectral sequence as a process successively approximating the graded parts $\gr_p H_n(C)$ of the filtered total homology ${\color{blue}H_n(C)}$:
\[
  ({\color{darkgreen}A_n},{\color{brown}R_n}) \leadsto
  ({\color{darkgreen}H_n(A)},{\color{brown}H_n(R)}) \leadsto
  ({\color{darkgreen}\coker}({\color{red}\partial_*}),{\color{brown}\ker}({\color{red}\partial_*})).
\]
The approximation is achieved by successively taking deeper inter-level interaction into account.

Finally one can ask if the spectral sequence above captured all the information in the long exact sequence. The answer is \emph{no}. The long exact sequence additionally contains the short exact sequence
\begin{equation}\label{extension}
  0 \xleftarrow{} {\color{brown}\ker}({\color{red}\partial_*})
  \xleftarrow{\nu_*} {\color{blue}H_n(C)} \xleftarrow{\iota_*}
  {\color{darkgreen}\coker}({\color{red}\partial_*}) \xleftarrow{} 0,
\end{equation}
explicitly describing the total homology ${\color{blue}H_n(C)}$ as an extension of its graded parts
${\color{darkgreen}\coker}({\color{red}\partial_*})$ and ${\color{brown}\ker}({\color{red}\partial_*})$.

Looking to what happens inside the subobject lattice of ${\color{blue}C_n}$ during the approximation process will help understanding how to remedy this defect.

\begin{figure}[htb]
  \begin{minipage}[c]{1\linewidth}
    \centering
    \psfrag{$C_n$}{$\color{blue}C_n$}
    \psfrag{$Z_n(R)$}{$\color{brown}Z_n(R)$}
    \psfrag{$B_n(R)$}{$\color{brown}B_n(R)$}
    \psfrag{$A_n$}{$\color{darkgreen}A_n$}
    \psfrag{$Z_n(A)$}{$\color{darkgreen}Z_n(A)$}
    \psfrag{$B_n(A)$}{$\color{darkgreen}B_n(A)$}
    \psfrag{$Z_n(C)$}{$Z_n(C)$}
    \psfrag{$B_n(C)$}{$B_n(C)$}
    \psfrag{$H_n(C)$}{$\color{blue}H_n(C)$}
    \psfrag{$ker$}{$\cong{\color{brown}\ker}({\color{red}\partial_*})$}
    \psfrag{$coker$}{$\cong{\color{darkgreen}\coker}({\color{red}\partial_*})$}
    \includegraphics[width=0.5\textwidth]{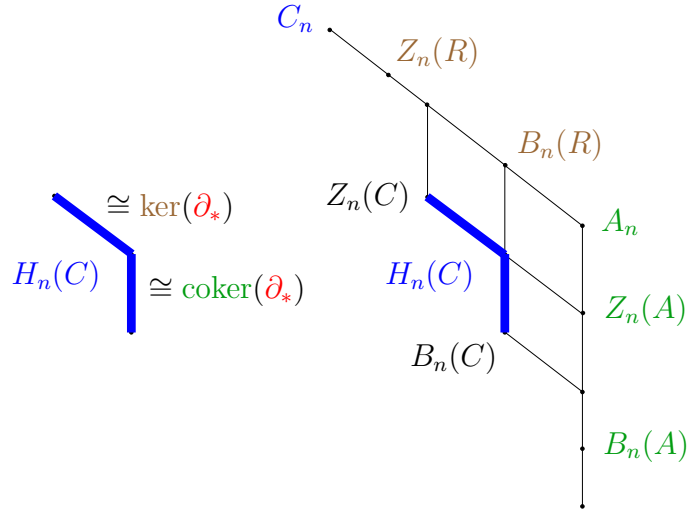}
\end{minipage}
 \caption{The $2$-step filtration $0 \leq \colorA \leq \colorC$ and the induced
   $2$-step filtration on $\color{blue}H_*(C)$}
  \label{LongExactSeq}
\end{figure}

Figure~\ref{LongExactSeq} shows the $n$-th object ${\color{blue}C_n}$ in the chain complex together with the subobjects that define the different homologies: ${\color{brown}H_n(R)}:={\color{brown}Z_n(R)}/{\color{brown}B_n(R)}$, ${\color{darkgreen}H_n(A)}:={\color{darkgreen}Z_n(A)}/{\color{darkgreen}B_n(A)}$, and ${\color{blue}H_n(C)}:={\color{blue}Z_n(C)}/{\color{blue}B_n(C)}$. Here we replaced ${\color{brown}Z_n(R)}$ and ${\color{brown}B_n(R)}$ by their full preimages in ${\color{blue}C_n}$ under the canonical epimorphism ${\color{blue}C_n} \xrightarrow{\nu} {\color{brown}R_n}:={\color{blue}C_n}/{\color{darkgreen}A_n}$.

\begin{figure}[htb]
  \begin{minipage}[c]{0.4\linewidth}
    \centering
    \psfrag{$C_n$}{$\color{blue}C_n$}
    \psfrag{$E^0_{1,n-1}:=R_n$}{${\color{brown}E^0_{1,n-1}}={\color{brown}R_n}$}
    \psfrag{$E^0_{0,n}:=A_n$}{${\color{darkgreen}E^0_{0,n}}={\color{darkgreen}A_n}$}
    \psfrag{$H_n(C)$}{$\color{blue}H_n(C)$}
    \includegraphics[width=0.4\textwidth]{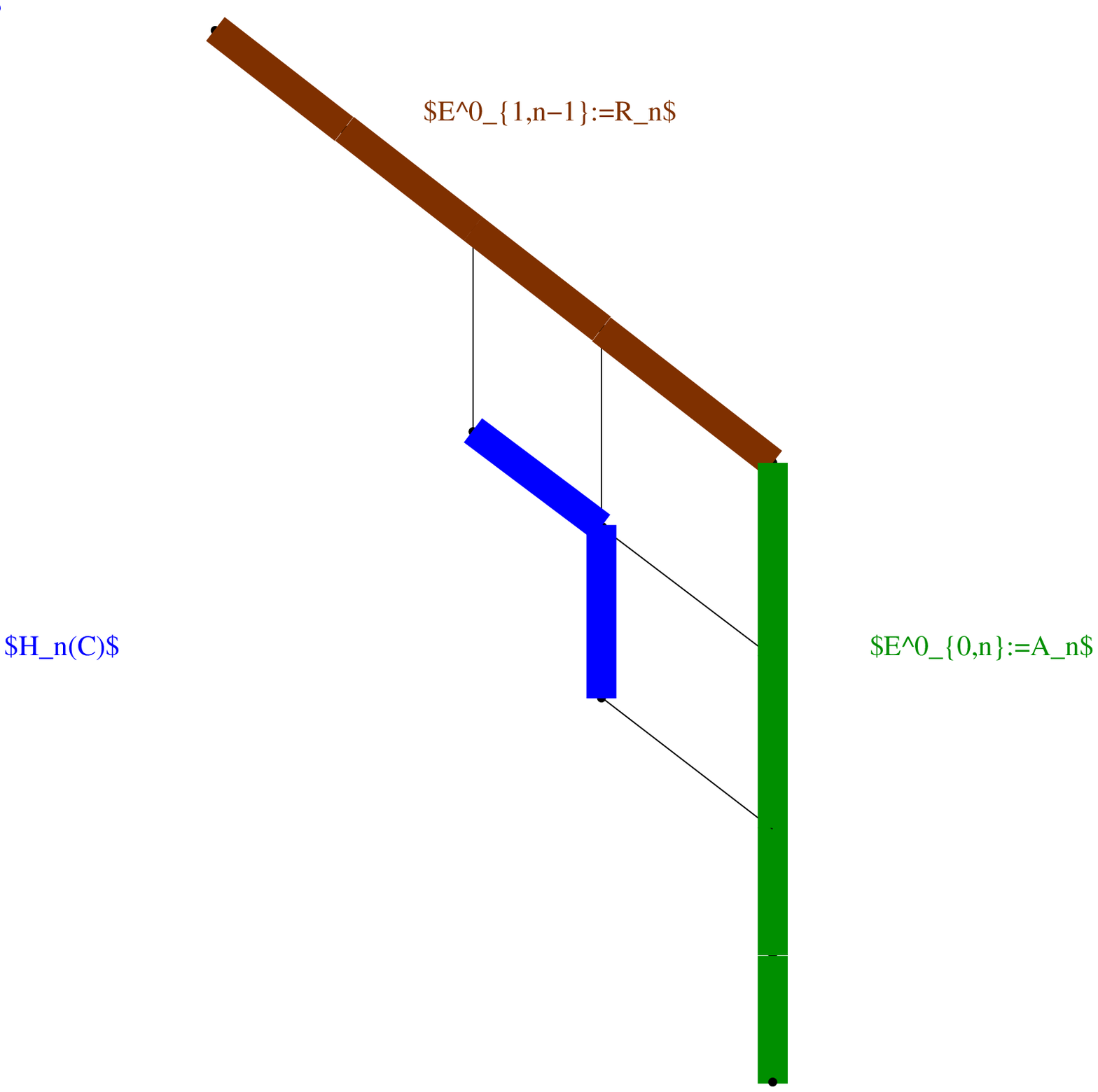}
    \caption{$E^0$}
    \label{E^0}
  \end{minipage}
  \quad $\leadsto$
  \begin{minipage}[c]{0.4\linewidth}
    \centering
    \psfrag{$C_n$}{$\color{blue}C_n$}
    \psfrag{$A_n$}{$\color{darkgreen}A_n$}
    \psfrag{$E^1_{1,n-1}:=H_n(R)$}{${\color{brown}E^1_{1,n-1}}={\color{brown}H_n(R)}$}
    \psfrag{$E^1_{0,n}:=H_n(A)$}{${\color{darkgreen}E^1_{0,n}}={\color{darkgreen}H_n(A)}$}
    \psfrag{$H_n(C)$}{$\color{blue}H_n(C)$}
    \includegraphics[width=0.4\textwidth]{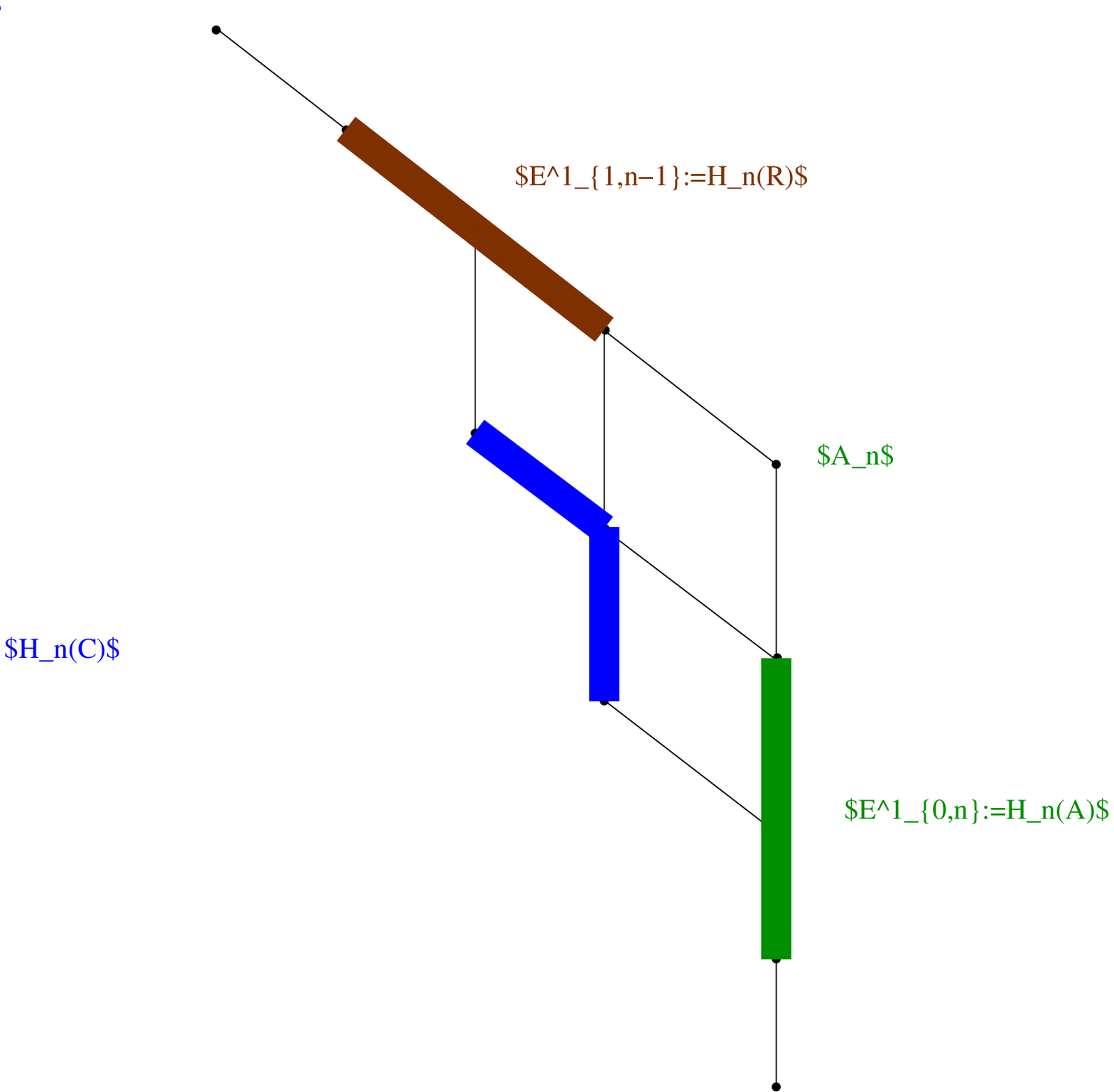}
    \caption{$E^1$}
    \label{E^1}
  \end{minipage}
  \quad \quad $\leadsto$
  \begin{minipage}[c]{0.4\linewidth}
    \vspace{0.5cm}
    \centering
    \psfrag{$C_n$}{$\color{blue}C_n$}
    \psfrag{$A_n$}{$\color{darkgreen}A_n$}
    \psfrag{$E^2_{1,n-1}$}{${\color{brown}E^2_{1,n-1}}={\color{brown}\ker}({\color{red}\partial_*})$}
    \psfrag{$E^2_{0,n}$}{${\color{darkgreen}E^2_{0,n}}={\color{darkgreen}\coker}({\color{red}\partial_*})$}
    \psfrag{$H_n(C)$}{$\color{blue}H_n(C)$}
    \includegraphics[width=0.4\textwidth]{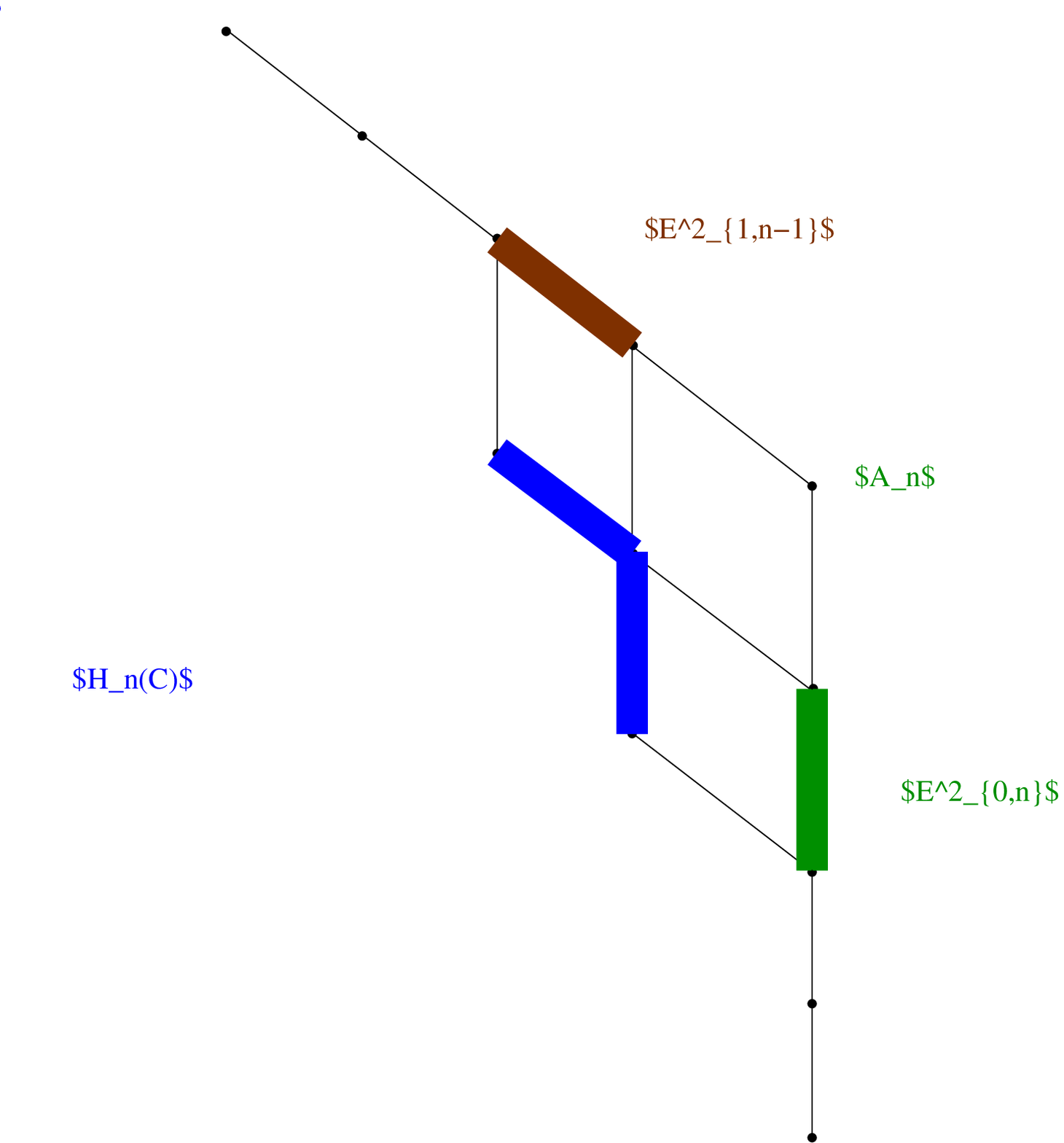}
    \caption{$E^2=E^\infty$}
    \label{E^2}
  \end{minipage}
  \\
  \vspace{0.5cm}
  The approximation process of the graded parts of ${\color{blue}H_n(C)}$
  \label{E}
\end{figure}

Figures~\ref{E^0}-\ref{E^2} show how the graded parts of ${\color{blue}H_n(C)}$ get successively approximated by the objects in the spectral sequence $E^r_{pq}$, naturally identified with certain subfactors of ${\color{blue}C_n}$ for $n=p+q$. Figure~\ref{E^2} proves that the second isomorphism theorem provides \emph{canonical} isomorphisms between the graded parts of the total homology ${\color{blue}H_n(C)}$ and the objects ${\color{brown}E^\infty_{1,n-1}}={\color{brown}E^2_{1,n-1}}$ and ${\color{darkgreen}E^\infty_{0,n}}={\color{darkgreen}E^2_{0,n}}$ of the stable sheet. And modulo these natural isomorphisms Figure~\ref{E^2} further suggests that knowing how to identify ${\color{brown}E^\infty_{1,n-1}}$ and ${\color{darkgreen}E^\infty_{0,n}}$ with the indicated subfactors of ${\color{blue}C_n}$ will suffice to explicitly construct the extension (\ref{extension}) in the form
\begin{equation}\label{extensionE}
  0 \xleftarrow{} {\color{brown}E^\infty_{1,n-1}}
  \xleftarrow{} {\color{blue}H_n(C)} \xleftarrow{}
  {\color{darkgreen}E^\infty_{0,n}} \xleftarrow{} 0.
\end{equation}
But since we cannot use maps to identify objects with subfactors of other objects we are lead to introduce the notion of \textbf{generalized maps} in the next Section. Roughly speaking, this notion enables us to interpret the pairs of horizontal arrows in Figure~\ref{Emb} as \textbf{generalized embeddings}.

\begin{figure}[htb]
  \begin{minipage}[c]{1.1\linewidth}
    \centering
    \psfrag{$C_n$}{$\color{blue}C_n$}
    \psfrag{$A_n$}{$\color{darkgreen}A_n$}
    \psfrag{$E^2_{1,n-1}$}{${\color{brown}E^\infty_{1,n-1}}$}
    \psfrag{$E^2_{0,n}$}{${\color{darkgreen}E^\infty_{0,n}}$}
    \psfrag{$H_n(C)$}{$\color{blue}H_n(C)$}
    \includegraphics[width=0.4\textwidth]{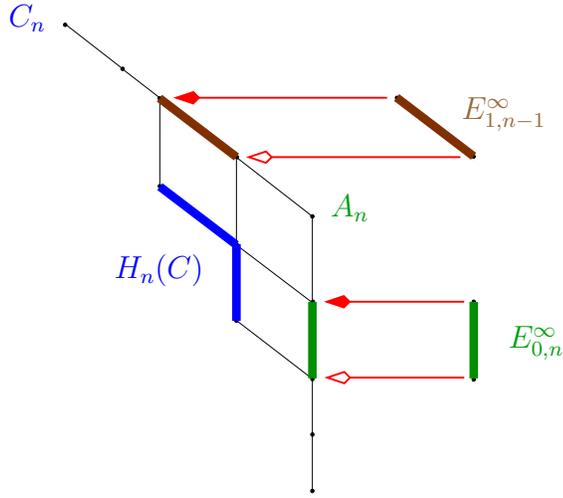}
  \end{minipage}
  \caption{The generalized embeddings}
  \label{Emb}
\end{figure}

\section{Generalized maps}\label{genmor}

A morphism between two objects (modules, complexes, \ldots) induces a map between their lattice of subobjects, and the \textbf{homomorphism theorem} implies that this map gives rise to a bijective correspondence between the subobjects of the target lying in the image and those subobjects of the source containing the kernel. This motivates the visualization in Figure~\ref{Mor} of a morphism $T\xleftarrow{\color{red}\phi}S$ with source $S$ and target $T$. The homomorphism theorem states that the morphism ${\color{red}\phi}$, indicated by the horizontal pair of arrows in Figure~\ref{Mor}, maps $S/\ker({\color{red}\phi})$ onto the \emph{subobject} $\img({\color{red}\phi})$ in a structure-preserving way. In this sense, the exact ladder of morphisms in (\ref{LES}) visualizes part of the long exact homology sequence.
\begin{figure}[htb]
  \begin{minipage}[c]{0.8\linewidth}
    \centering
    \psfrag{$T$}{$T$}
    \psfrag{$S$}{$S$}
    \psfrag{$phi$}{$\color{red}\phi$}
    \psfrag{$im(phi)$}{$\img{\color{red}\phi}$}
    \psfrag{$ker(phi)$}{$\ker{\color{red}\phi}$}
    \includegraphics[width=0.4\textwidth]{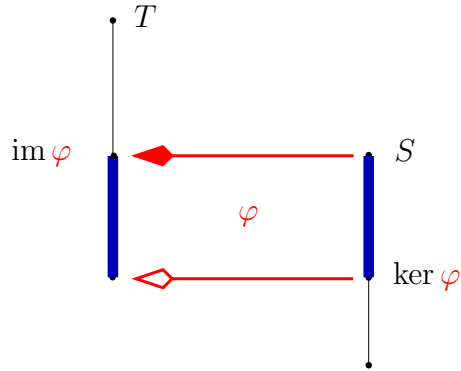}
  \end{minipage}
  \caption{The homomorphism theorem}
  \label{Mor}
\end{figure}

The simplest motivation for the notion of a generalized morphism $T\xleftarrow{\color{red}\psi}S$ is the desire to give sense to the picture in Figure~\ref{Gen} ``mapping'' a quotient of $S$ onto a \emph{subfactor} of $T$.
\begin{figure}[htb]
  \begin{minipage}[c]{0.8\linewidth}
    \centering
    \psfrag{$T$}{$T$}
    \psfrag{$S$}{$S$}
    \psfrag{$L$}{$L$}
    \psfrag{$psi$}{$\color{red}\psi$}
    \psfrag{$im(psi)$}{$\img{\color{red}\psi}$}
    \psfrag{$Im(psi)$}{$\Img{\color{red}\psi}$}
    \psfrag{$ker(psi)$}{$\ker{\color{red}\psi}$}
    \includegraphics[width=0.4\textwidth]{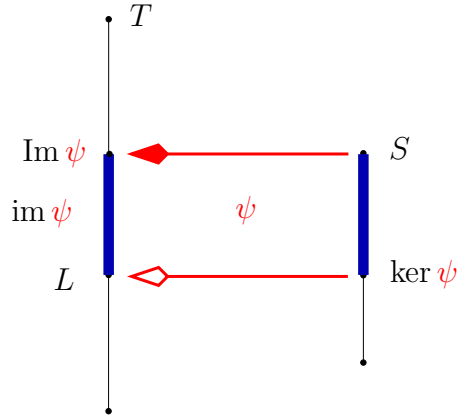}
  \end{minipage}
  \caption{A generalized morphism}
  \label{Gen}
\end{figure}

\begin{defn}[Generalized morphism]
  Let $S$ and $T$ be two objects in an abelian category (of modules over some ring). A \textbf{generalized morphism} $\psi$ with source $S$ and target $T$ is a pair of morphisms $(\bar{\psi},\imath)$, where $\imath$ is a morphism from some third object $F$ to $T$ and $\bar{\psi}$ is a morphism from $S$ to $\coker\imath=T/\img(\imath)$. We call $\bar{\psi}$ the morphism \textbf{associated} to $\psi$ and  $\imath$ the \textbf{morphism aid} of $\psi$ and denote it by $\Aid\psi$. Further we call $L:=\img\imath\leq T$ the \textbf{morphism aid subobject}. Two generalized morphisms $(\bar{\psi},\imath)$ and $(\bar{\phi},\jmath)$ with ($\img \imath = \img \jmath$ and) $\bar{\psi}=\bar{\phi}$ will be identified.
\end{defn}

Philosophically speaking, this definition frees one from the ``conservative'' standpoint of viewing $\psi$ as morphism to the quotient $T/\img\imath$. Instead it allows one to view $\psi$ as a ``morphism'' to the full object $T$ by directly incorporating $\imath$ in the very definition of $\psi$. The intuition behind the notion ``morphism aid'' (resp.\  ``morphism aid subobject'') is that $\imath$ (resp.\  $L=\img\imath$) \emph{aids} $\psi$ to become a (well-defined) morphism. Figure~\ref{GenMor} visualizes the generalized morphism $\psi$ as a pair $(\bar{\psi},\imath)$.

\begin{figure}[htb]
  \begin{minipage}[c]{1.1\linewidth}
    \centering
    \psfrag{$F$}{$F$}
    \psfrag{$T$}{$T$}
    \psfrag{$S$}{$S$}
    \psfrag{$T/im(alpha)$}{$T/\img \imath$}
    \psfrag{$barpsi$}{$\bar{\psi}$}
    \psfrag{$pi$}{$\pi_\imath$}
    \psfrag{$alpha$}{$\imath$}
    \psfrag{$im(psi)$}{$\img{\color{red}\psi}$}
    \psfrag{$im(barpsi)$}{$\img\bar{\psi}$}
    \psfrag{$im(alpha)$}{$L=\img\imath$}
    \psfrag{$ker(barpsi)$}{$\ker\bar{\psi}$}
    \psfrag{$pi^{-1}(im(barpsi))$}{$\pi_\imath^{-1}(\img \bar{\psi}) =: \Img{\color{red}\psi}$}
    \includegraphics[width=0.6\textwidth]{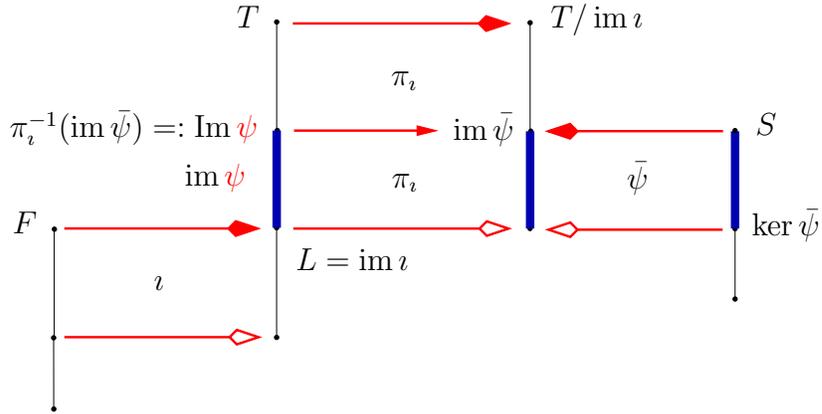}
  \end{minipage}
  \caption{The morphism aid $\imath$ and the associated morphism $\bar{\psi}$}
  \label{GenMor}
\end{figure}

Note that replacing $\imath$ by a morphism with the same image does not alter the generalized morphism. We will therefore often write $(\bar{\psi},L)$ for the generalized morphism $(\bar{\psi},\imath)$, where $\imath$ is any morphism with $\img \imath = L \leq T$. The most natural choice would be the embedding $\imath:L \to T$. Figure~\ref{Gen} visualizes the generalized morphism $\psi$ as a pair $(\bar{\psi},L)$. It also reflects the idea behind the definition more than the ``expanded'' Figure~\ref{GenMor} does.

If $L=\img\imath$ vanishes, then $\psi$ is nothing but the (ordinary) morphism $\bar{\psi}$. Conversely, any morphism can be viewed as a generalized morphism with trivial morphism aid subobject $L=0$.

\begin{defn}[Terminology for generalized morphisms]
  Let  $\psi=(\bar{\psi},\imath): S \to T$ be a generalized morphism. Define the \textbf{kernel} $\ker(\psi):=\ker \bar{\psi}$, the kernel of the associated map. If $\pi_\imath$ denotes the natural epimorphism $T\to T/\img\imath$, then define the \textbf{combined image} $\Img\psi$ to be the \emph{submodule} $\pi_\imath^{-1}(\img \bar{\psi})$ of $T$. In general it differs from the \textbf{image} $\img\psi$ which is defined as the \emph{subfactor} $\Img\psi/\img \imath$ of $T$ (cf.~Figure~\ref{GenMor}). We call $\psi$ a \textbf{generalized monomorphism} (resp.\  \textbf{generalized epimorphism}, \textbf{generalized isomorphism}) if the associated map $\bar{\psi}$ is a monomorphism (resp.\ epimorphism, isomorphism).
\end{defn}

Sometimes we use the terminology \textbf{generalized map} instead of generalized morphism and \textbf{generalized embedding} instead of generalized monomorphism, especially when the abelian category is a category of modules (or complexes of modules, etc.).

\bigskip
As a first application of the notion of generalized embeddings we state the following definition, which is central for this work.

\begin{defn}[Filtration system]\label{system}
Let $\mathcal{I}=(p_0,\ldots,p_{m-1})$ be a finite interval in $\Z$, i.e.\  $p_{i+1}=p_i+1$. \\
A finite sequence of generalized embeddings $\psi_p=(\bar{\psi}_p,L_p)$, $p\in\mathcal{I}$ with common target $M$ is called an \textbf{ascending $m$-filtration system} of $M$ if
\begin{enumerate}
  \item $\psi_{p_0}$ is an ordinary monomorphism, i.e.\  $L_{p_0}$ vanishes;
  \item $L_p=\Img\psi_{p-1}$, for $p=p_1,\ldots,p_{m-1}$;
  \item $\psi_{p_{m-1}}$ is a generalized isomorphism, i.e.\  $\Img\psi_{p_{m-1}}=M$.
\end{enumerate}
\begin{figure}[htb]
  \begin{minipage}[c]{1.1\linewidth}
    \centering
    \psfrag{$psi_0$}{$\psi_{p_0}$}
    \psfrag{$psi_1$}{$\psi_{p_1}$}
    \psfrag{$psi_{m-2}$}{$\psi_{p_{m-2}}$}
    \psfrag{$psi_{m-1}$}{$\psi_{p_{m-1}}$}
    \psfrag{$L_1$}{$L_{p_1}$}
    \psfrag{$L_2$}{$L_{p_2}$}
    \psfrag{$L_{m-2}$}{$L_{p_{m-2}}$}
    \psfrag{$L_{m-1}$}{$L_{p_{m-1}}$}
    \psfrag{$M$}{$M$}
    \includegraphics[width=0.55\textwidth]{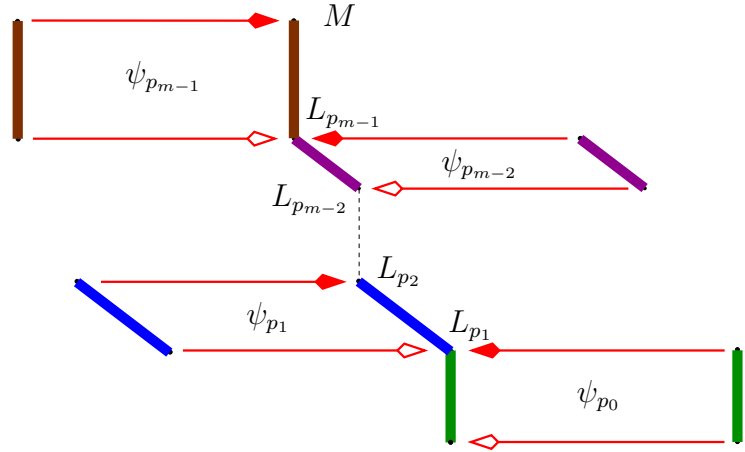}
  \end{minipage}
  \caption{An ascending $m$-filtration system}
  \label{System}
\end{figure}
A finite sequence of generalized embeddings $\psi^p=(\bar{\psi}^p,L^p)$, $p\in\mathcal{I}$ with common target $M$ is called a \textbf{descending $m$-filtration system} of $M$ if
\begin{enumerate}
  \item $\psi^{p_0}$ is a generalized isomorphism, i.e.\  $\Img\psi^{p_0}=M$;
  \item $L^p=\Img\psi^{p+1}$, for $p=p_0,\ldots,p_{m-2}$;
  \item $\psi^{p_{m-1}}$ is an ordinary monomorphism, i.e.\  $L^{p_{m-1}}$ vanishes.
\end{enumerate}
We say $(\psi_p)$ \textbf{computes} a given filtration $(F_p M)$ if $\Img \psi_p = F_p M$ for all $p$.
\end{defn}

\bigskip
Now we come to the definition of the basic operations for generalized morphisms. Two generalized maps $\psi=(\bar{\psi},\imath)$ and $\phi=(\bar{\phi},\jmath)$ are summable only if $\img\imath=\img\jmath$ and we set $\psi\pm\phi:=(\bar{\psi}\pm\bar{\phi},\imath)$.

The following notational convention will prove useful: It will often happen that one wants to alter a generalized morphism $\psi=(\bar{\psi},L_\psi)$ with target $T$ by replacing $L_\psi$ with a larger subobject $L$, i.e. a subobject $L \leq T$ containing $L_\psi$. We will sloppily write $\widetilde{\psi}=(\bar{\psi},L)$, where $\bar{\psi}$ now stands for the composition of $\bar{\psi}$ with the natural epimorphism $T/L_\psi \to T/L$. We will say that $\psi$ was \textbf{coarsened} to $\widetilde{\psi}$ to refer to the passage from $\psi=(\bar{\psi},L_\psi)$ to $\widetilde{\psi}=(\bar{\psi},L)$ with $L_\psi\leq L \leq T$. As Figure~\ref{Coarsen} shows, coarsening $\psi$ might very well enlarge its combined image $\Img\psi$. The word ``coarse'' refers to the fact that the image $\img\widetilde{\psi}$ is naturally isomorphic to a \emph{quotient} of $\img\psi$, and Figure~\ref{Coarsen} shows that this natural isomorphism is given by the second isomorphism theorem. We say that the coarsening $\widetilde{\psi}=(\bar{\psi},L)$ of $\psi=(\bar{\psi},L_\psi)$ is \textbf{effective}, if $\Img\psi \cap L = L_\psi$. Figure~\ref{Coarsen} shows that in this case the images $\img\psi$ and $\img\widetilde{\psi}$ are naturally \emph{isomorphic}.

\begin{figure}[htb]
  \begin{minipage}[c]{1.1\linewidth}
    \centering
    \psfrag{$T$}{$T$}
    \psfrag{$S$}{$S$}
    \psfrag{$L$}{$L$}
    \psfrag{$L_psi$}{$L_\psi$}
    \psfrag{$psi$}{$\psi$}
    \psfrag{$phi$}{$\widetilde{\psi}$}
    \psfrag{$Im(psi)$}{$\Img\psi$}
    \psfrag{$Im(phi)$}{$\Img\widetilde{\psi}$}
    \psfrag{$im(psi)$}{$\img\psi$}
    \psfrag{$im(phi)$}{$\img\widetilde{\psi}$}
    \psfrag{$ker(psi)$}{$\ker\psi$}
    \psfrag{$ker(phi)$}{$\ker\widetilde{\psi}$}
    \includegraphics[width=0.6\textwidth]{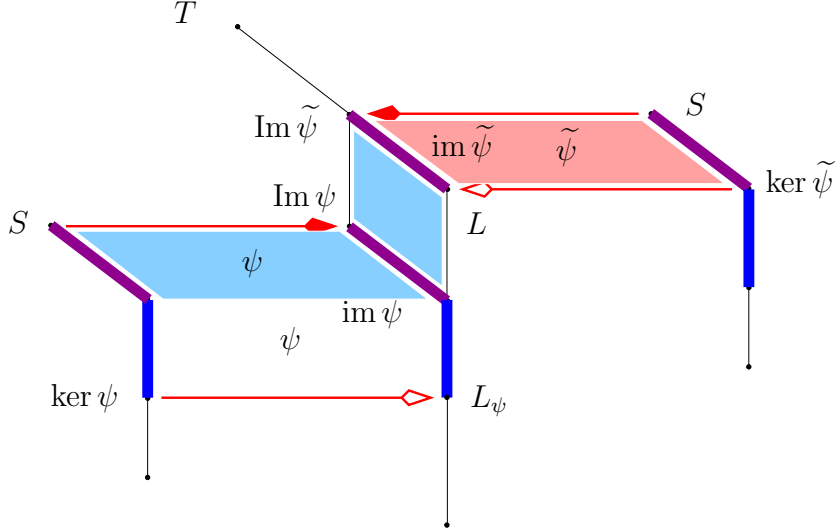}
  \end{minipage}
  \caption{Coarsening the generalized map $\psi=(\bar{\psi},K)$ to $\widetilde{\psi}=(\bar{\psi},L)$}
  \label{Coarsen}
\end{figure}

For the composition $\psi\circ\phi$ of $S_\phi \xrightarrow{\phi} T_\phi=S_\psi \xrightarrow{\psi} T_\psi$ follow the filled area in Figure~\ref{Composition} from left to right.
\begin{figure}[htb]
  \begin{minipage}[c]{1.06\linewidth}
    \centering
    \psfrag{$S_phi$}{$S_\phi$}
    \psfrag{$T_phi=S_psi$}{$T_\phi=S_\psi$}
    \psfrag{$T_psi$}{$T_\psi$}
    \psfrag{$T_psi$}{$T_\psi$}
    \psfrag{$phi$}{$\phi$}
    \psfrag{$ker(phi)$}{$\ker\phi$}
    \psfrag{$ker(tildephi)$}{$\ker \psi\circ\phi=\ker\widetilde{\phi}$}
    \psfrag{$Im(phi)$}{$\Img\phi$}
    \psfrag{$im(phi)$}{$\img\phi$}
    \psfrag{$im(jota)$}{$\img\jmath$}
    \psfrag{$K$}{$K$}
    \psfrag{$psi$}{$\psi$}
    \psfrag{$ker(psi)$}{$\ker\psi$}
    \psfrag{$Im(psi)$}{$\Img\psi$}
    \psfrag{$im(psi phi)$}{$\color{magenta}\img\psi\circ\phi$}
    \psfrag{$Im(psi phi)$}{$\color{magenta}\Img\psi\circ\phi$}
    \psfrag{$L$}{$\color{magenta} L := \pi_\imath^{-1}(\img (\bar{\psi} \circ \jmath))$}
    \psfrag{$im(iota)$}{$\img\imath$}
    \includegraphics[width=0.7\textwidth]{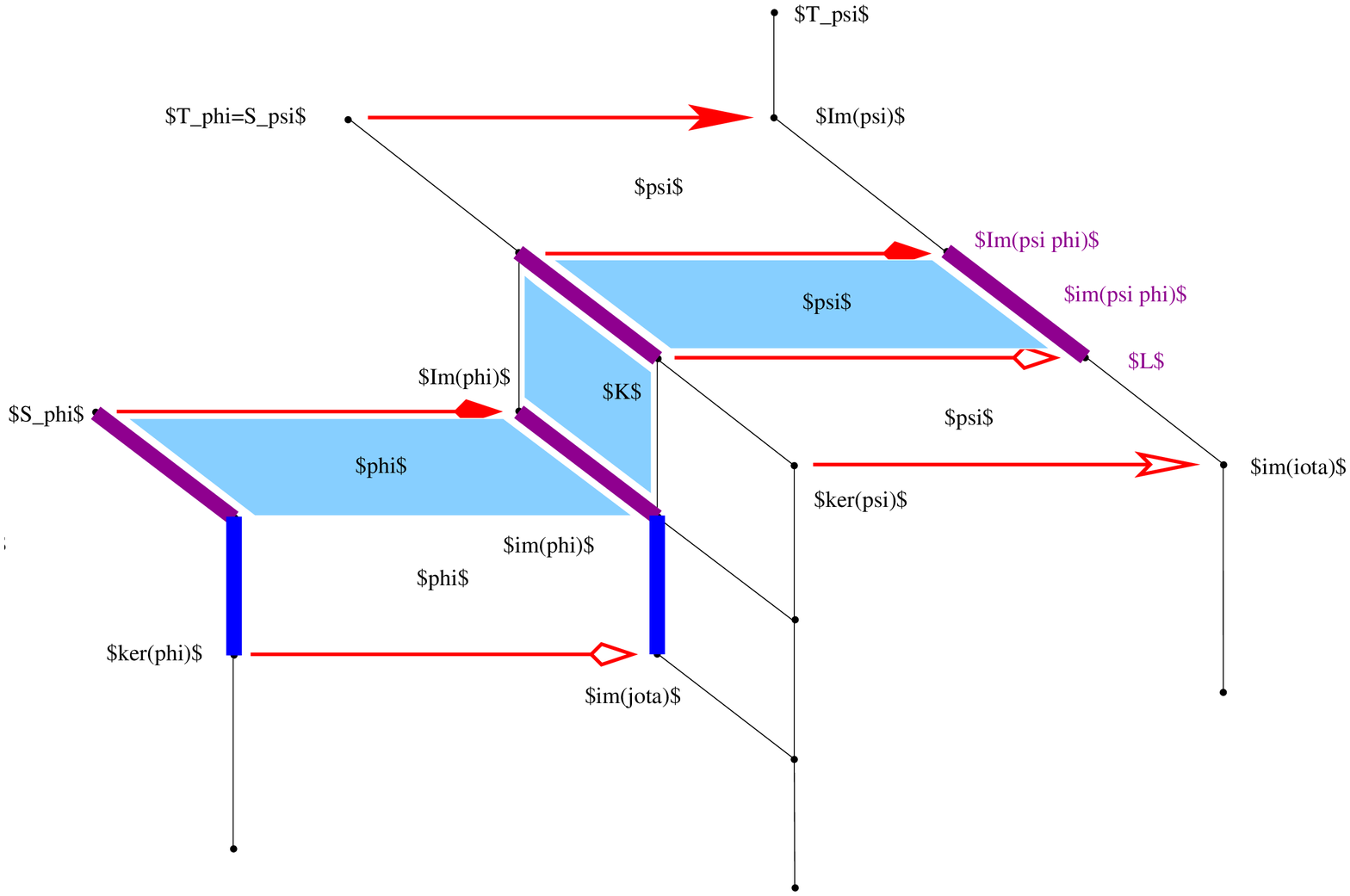}
  \end{minipage}
  \caption{The composition $\psi\circ\phi$}
  \label{Composition}
\end{figure}

Formally, first coarsen $\phi=(\bar{\phi},\jmath) \to \widetilde{\phi}=(\bar{\phi},K)$, where
\[
  K:=\img\jmath + \ker\psi \leq T_\phi.
\]
Then coarsen $\psi=(\bar{\psi},\imath) \to \widetilde{\psi}=(\bar{\psi},L)$, where
\[
  L := \pi_\imath^{-1}(\img (\bar{\psi} \circ \jmath)) =  \pi_\imath^{-1}(\bar{\psi}(K)) \leq T_\psi
\]
and $\pi_\imath$ as above. Now set
\[
   \psi\circ\phi := (\bar{\psi}\circ \bar{\phi},L).
\]
Note that $\ker \psi\circ \phi = \ker\widetilde{\phi}$.

\smallskip
Finally we define the division $\beta^{-1} \circ \gamma$ of two generalized maps $S_\gamma\xrightarrow{\gamma} T \xleftarrow{\beta}S_\beta$ under the conditions of the next definition.

\begin{defn}[The lifting condition]
Let $\gamma=(\bar{\gamma},L_\gamma)$ and $\beta=(\bar{\beta},L_\beta)$ be two generalized morphisms with the same target $N$.
\[
  \xymatrix{
  M' \ar[rd]^\gamma & \\
  N' \ar[r]_\beta & N.
  }
\]
Consider the \textbf{common coarsening} of the generalized maps $\beta$ and $\gamma$, i.e. the generalized maps $\widetilde{\beta}:=(\bar{\beta},L)$ and $\widetilde{\gamma}:=(\bar{\gamma},L)$, where $L=L_\gamma+L_\beta \leq N$. We say $\beta$ \textbf{lifts} $\gamma$ (or \textbf{divides} $\gamma$) if the following two conditions are satisfied:
\begin{enumerate}
  \item[\textbf{(im)}] The combined image of $\widetilde{\beta}$ contains the combined image of $\widetilde{\gamma}$:
  \[
  \Img\widetilde{\gamma} \leq \Img\widetilde{\beta}.
  \]
  \item[\textbf{(eff)}] The coarsening $\gamma \to \widetilde{\gamma}$ is effective, i.e.\  $\Img\gamma \cap L = L_\gamma$.
\end{enumerate}
\end{defn}

We will refer to $\widetilde{\gamma}$ as \textbf{the effective coarsening of $\gamma$ with respect to $\beta$}.
The following lemma justifies this definition. Both the definition and the lemma are visualized in Figure~\ref{Lemma}. To state the lemma one last notion is needed: Define two generalized morphisms $\psi=(\bar{\psi},L_\psi)$ and $\phi=(\bar{\phi},L_\phi)$ to be \textbf{equal up to effective common coarsening} or \textbf{quasi-equal} if their common coarsenings  $\widetilde{\psi}:=(\overline{\psi},L)$ and $\widetilde{\phi}:=(\overline{\phi},L)$ coincide \emph{and} are \emph{both} effective. We write $\psi\triangleq\phi$.

\begin{figure}[htb]
  \begin{minipage}[c]{1\linewidth}
    \centering
    \psfrag{$N$}{$N$}
    \psfrag{$M'$}{$M'$}
    \psfrag{$N'$}{$N'$}
    \psfrag{$L$}{$L$}
    \psfrag{$Im(alpha)$}{$\Img\alpha$}
    \psfrag{$im(alpha)$}{$\img\alpha$}
    \psfrag{$L_alpha$}{$L_\alpha$}
    \psfrag{$L_beta$}{$L_\beta$}
    \psfrag{$L_gamma$}{$L_\gamma$}
    \psfrag{$Im(tildebeta)$}{$\Img\widetilde{\beta}$}
    \psfrag{$Im(tildegamma)$}{$\Img\widetilde{\gamma}$}
    \psfrag{$beta$}{$\beta$}
    \psfrag{$ker(beta)$}{$\ker\beta$}
    \psfrag{$Im(beta)$}{$\Img\beta$}
    \psfrag{$gamma$}{$\gamma$}
    \psfrag{$ker(gamma)$}{$\ker\gamma$}
    \psfrag{$Im(gamma)$}{$\Img\gamma$}
    \psfrag{$phi$}{$\widetilde{\psi}$}
    \psfrag{$Im(psi)$}{$\Img\psi$}
    \psfrag{$Im(phi)$}{$\Img\widetilde{\psi}$}
    \psfrag{$im(psi)$}{$\img\psi$}
    \psfrag{$im(phi)$}{$\img\widetilde{\psi}$}
    \psfrag{$ker(psi)$}{$\ker\psi$}
    \psfrag{$ker(phi)$}{$\ker\widetilde{\psi}$}
    \includegraphics[width=0.8\textwidth]{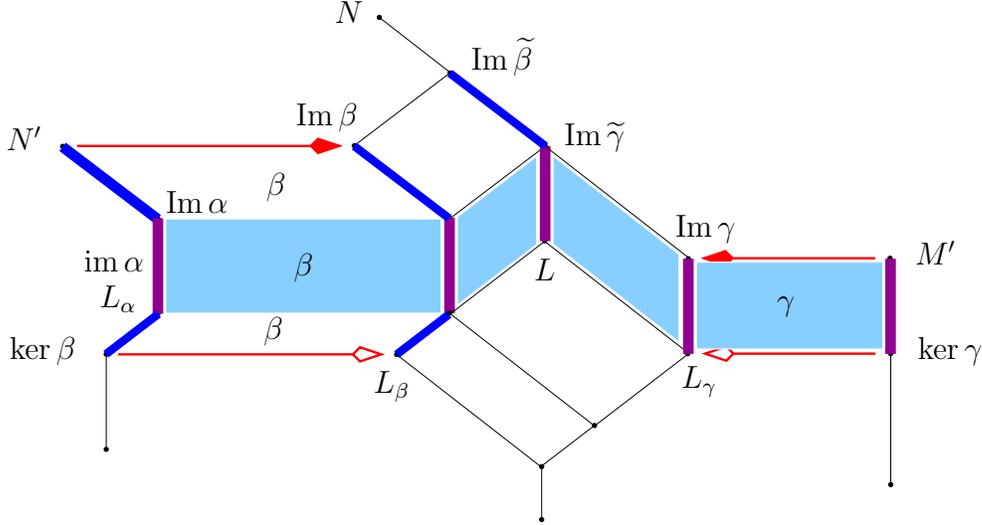}
  \end{minipage}
  \caption{The lifting condition and the lifting lemma}
  \label{Lemma}
\end{figure}

\begin{lemma}[The lifting lemma]\label{lifting_lemma}
Let $\gamma=(\bar{\gamma},L_\gamma)$ and $\beta=(\bar{\beta},L_\beta)$ be two generalized morphisms with the same target $N$. Suppose that $\beta$ lifts $\gamma$. Then there exists a generalized morphism $\alpha:M' \to N'$ with $\beta\circ\alpha \triangleq \gamma$,
\[
  \xymatrix{
  M' \ar[rd]^\gamma \ar[d]^\alpha & \\
  N' \ar[r]_\beta & N.
  }
\]
i.e.\  $\beta\circ\alpha$ is equal to $\gamma$ up to effective common coarsening. $\alpha$ is called \textbf{a lift} of $\gamma$ along $\beta$. \\
Further let $\widetilde{\gamma}:=(\bar{\gamma},L_{\widetilde{\gamma}})$ be the effective coarsening of $\gamma$ with respect to $\beta$, i.e.\  $L_{\widetilde{\gamma}}=L=L_\gamma+L_\beta$. Then there exists a \emph{unique} lift $\alpha=(\bar{\alpha},L_\alpha)$ satisfying
\begin{enumerate}
  \item[(a)] $\Img\alpha = \bar{\beta}^{-1}(\Img\widetilde{\gamma})$ and
  \item[(b)] $L_\alpha = \bar{\beta}^{-1}(L_{\widetilde{\gamma}})$.
\end{enumerate}
This $\alpha$ is called \textbf{the lift} of $\gamma$ along $\beta$, or \textbf{the quotient} of $\gamma$ by $\beta$ and is denoted by $\beta^{-1}\circ \gamma$ or by ${\gamma}/{\beta}$.
\end{lemma}
\begin{proof}
  The subobject lattice(s) in Figure~\ref{Lemma} describes the most general setup imposed by conditions (im) and (eff), in the sense that all other subobject lattices of configurations satisfying these two conditions are at most degenerations of the one in Figure~\ref{Lemma}. Now to construct the unique $\alpha$ simply follow the filled area from right to left.
\end{proof}

The reader may have already noticed that the choice of the symbol $\triangleq$ for quasi-equality was motivated by Figure~\ref{Lemma}, with $L$ at the tip of the pyramid. The proof makes it clear that the lifting lemma is yet another incarnation of the homomorphism theorem.

\begin{rmrk}[Effective computability]\label{comp}
  Note that the lift $\alpha=(\bar{\alpha},L_\alpha)$ sees from $N'$ only its subfactor $N'/L_\alpha$. Replacing $N'$ by its subfactor $N'/L_\alpha$ turns $\beta$ into a generalized embedding, which we again denote by $\beta$. Now $\gamma$ and \emph{this} $\beta$ have effective common coarsenings $\widetilde{\gamma}=(\bar{\gamma},L)$ and $\widetilde{\beta}=(\bar{\beta},L)$, which see from $N$ only $N/L$, where $L=L_\gamma+L_\beta$. And modulo $L$ the generalized morphism $\widetilde{\gamma}$ becomes a morphism and the generalized embedding $\widetilde{\beta}$ becomes an (ordinary) embedding. So from the point of view of effective computations the setup can be reduced to the following situation: $\gamma:M'\to N$ is a morphism and $\beta:N' \to N$ is a \emph{monomorphism}. When $M'$, $N'$, and $N$ are finitely presented modules over a \textbf{computable ring} (cf.~Def.~\ref{compdef}) it was shown in \cite[Subsection~3.1.1]{BR} that in this case the unique morphism $\alpha:M'\to N$ is \textbf{effectively} computable.
\end{rmrk}

With the notion of a generalized embedding at our disposal we can finally give the horizontal arrows in Figure~\ref{Emb} a meaning. Now consider the three generalized embeddings $\iota:  \color{blue}H_n(C) \to \color{blue}C_n$, $\iota_0: {\color{darkgreen}E^\infty_{0,n}} \to \color{blue}C_n$, and $\iota_1: {\color{brown}E^\infty_{1,n-1}} \to 	\color{blue}C_n$ in Figure~\ref{Lift}. $\iota_p$ is called the \textbf{total embedding} of $E^\infty_{p,n-p}$.

\begin{figure}[htb]
  \begin{minipage}[c]{1.1\linewidth}
    \centering
    \psfrag{$C_n$}{$\color{blue}C_n$}
    \psfrag{$A_n$}{$\color{darkgreen}A_n$}
    \psfrag{$E^2_{1,n-1}$}{${\color{brown}E^\infty_{1,n-1}}$}
    \psfrag{$E^2_{0,n}$}{${\color{darkgreen}E^\infty_{0,n}}$}
    \psfrag{$H_n(C)$}{$\color{blue}H_n(C)$}
    \psfrag{$iota$}{$\iota$}
    \psfrag{$iota0$}{$\iota_0$}
    \psfrag{$iota1$}{$\iota_1$}
    \includegraphics[width=0.55\textwidth]{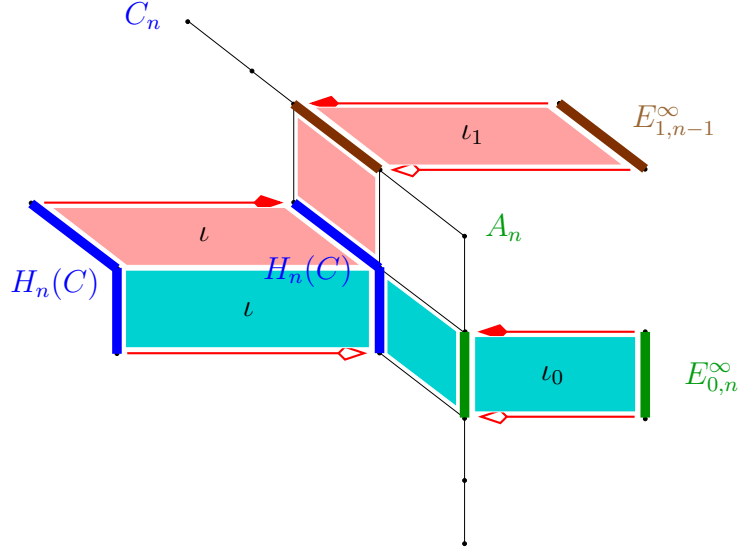}
  \end{minipage}
  \caption{$\iota$ lifts $\iota_0$ and $\iota_1$}
  \label{Lift}
\end{figure}

\begin{coro}\label{coro_2-filt}
The generalized embedding $\iota$ in Figure~\ref{Lift} lifts both total embeddings $\iota_0$ and $\iota_1$. Thus the two lifts $\epsilon_0 := {\iota_0}/{\iota}$ and $\epsilon_1 := {\iota_1}/{\iota}$ are generalized embeddings that form a filtration system of $H_n(C)$, visualized in Figure~\ref{2-Filtration}. More precisely, $\epsilon_0$ is an (ordinary) embedding and $\epsilon_1$ is a generalized isomorphism.
\end{coro}
\begin{figure}[htb]
  \begin{minipage}[c]{1.1\linewidth}
    \centering
    \psfrag{$E^2_{1,n-1}$}{${\color{brown}E^\infty_{1,n-1}}$}
    \psfrag{$E^2_{0,n}$}{${\color{darkgreen}E^\infty_{0,n}}$}
    \psfrag{$H_n(C)$}{$\color{blue}H_n(C)$}
    \psfrag{$epsilon0$}{$\epsilon_0 = \iota_0 /  \iota$}
    \psfrag{$epsilon1$}{$\epsilon_1 = \iota_1 /  \iota$}
    \includegraphics[width=0.55\textwidth]{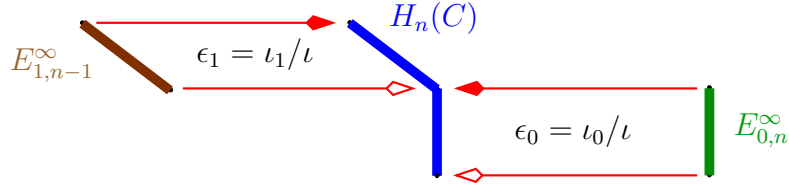}
  \end{minipage}
  \caption{The filtration of $H_n(C)$ given by the $2$-filtration system $\epsilon_0$, $\epsilon_1$}
  \label{2-Filtration}
\end{figure}
\begin{proof}
There are two obvious degenerations of the subobject lattice(s) in Figure~\ref{Lemma}, both leading to a sublattice of the lattice in Figure~\ref{Lift}, one for the pair $(\beta,\gamma)=(\iota,\iota_0)$ and the other for $(\beta,\gamma)=(\iota,\iota_1)$. In other words: Following the two filled areas from right to left constructs $\epsilon_0:= {\iota}^{-1}\circ{\iota_0}$ and $\epsilon_1:={\iota}^{-1}\circ{\iota_1}$.
\end{proof}

\begin{coro}[Generalized inverse]\label{geninv}
  Let $\psi:S \to T$ be a generalized epimorphism. Then there exists a \emph{unique} generalized epimorphism $\psi^{-1}:T \to S$, such that $\psi^{-1}\circ\psi = (\id_{S},\ker\psi)$ and $\psi\circ\psi^{-1} = (\id_{T},\Aid\psi)$. $\psi^{-1}$ is called the \textbf{generalized inverse} of $\psi$. In particular, if $\psi$ is an (ordinary) epimorphism, then $\psi^{-1}$ is a generalized isomorphism, and vice versa.
\end{coro}
\begin{proof}
  Since $\psi$ lifts $\id_T$ define $\psi^{-1}:=\id_T/\psi$.
\end{proof}

Rephrasing short exact sequences (also called $1$-extensions) in terms of $2$-filtration systems is now an easy application of this corollary. In particular, the information in the short exact sequence (\ref{extensionE}) is fully captured by the $2$-filtration system in Figure~\ref{2-Filtration}. This is last step of remedying the defect mentioned while introducing the short exact sequence (\ref{extension}) in Section~\ref{les}.

\section{Spectral sequences of filtered complexes}\label{filt}

Everything substantial already happened in Sections~\ref{les} and \ref{genmor}. Here we only show how the ideas already developed for $2$-filtrations and their $2$-step spectral sequences easily generalize to $m$-filtrations and their $m$-step spectral sequences.

We start by recalling the construction of the \textbf{spectral sequence associated to a filtered complex}. The exposition till Theorem~\ref{mainthm} closely follows \cite[Section~5.4]{WeiHom}. We also remain loyal to our use of subobject lattices as they are able to sum up a considerable amount of relations in one picture.

Consider a chain complex $C$ with (an ascending) filtration $F_p C$. The complementary degree $q$ and the total degree $n$ are dropped for better readability. Define the natural projection $F_pC \to F_p C/F_{p-1} C =: E^0_p$. It is elementary to check that the \textbf{subobjects of $r$-approximate cycles}
\[
  A^r_p := \ker (F_p C \to F_p C / F_{p-r} C) = \{ c \in F_p C \mid \partial c \in F_{p-r} C\}
\]
satisfy the relations of Figure~\ref{E_r}, with $Z^r_p := A^r_p + F_{p-1} C$, $B^r_p := \partial A^{r-1}_{p+(r-1)} + F_{p-1} C$, and $E^r_p := Z^r_p / B^r_p$. These definitions deviate a bit from those in \cite[Section~5.4]{WeiHom}. Here $Z^r_p$ and $B^r_p$ sit between $F_p C$ and $F_{p-1} C$. His $Z^r_p$ and $B^r_p$ are the projections under $\eta_p$ onto $E^0_p:=F_p C / F_{p-1} C$ of the ones here, and hence sit in the objects of the $0$-th sheet $E^0_p$. The subobject lattice in Figure~\ref{E_r} should by now be considered an old friend as it is ubiquitous throughout all our arguments.

\begin{figure}[htb]
  \begin{minipage}[c]{1.1\linewidth}
    \centering
    \psfrag{$F_p C$}{$F_p C$}
    \psfrag{$F_{p-1} C$}{$F_{p-1} C$}
    \psfrag{$E^r_p$}{$E^r_p$}
    \psfrag{$\iota^r_p$}{}
    \psfrag{$Z^r_p$}{$Z^r_p$}
    \psfrag{$B^r_p$}{$B^r_p$}
    \psfrag{$A^r_p$}{$A^r_p$}
    \psfrag{$A^{r-1}_{p-1}$}{$A^{r-1}_{p-1}$}
    \psfrag{$partial(A^{r-1})$}{$\partial A^{r-1}_{p+(r-1)}$}
    \psfrag{$partial(A^r)$}{$\partial A^r_{p-1+(r)}$}
   \includegraphics[width=0.4\textwidth]{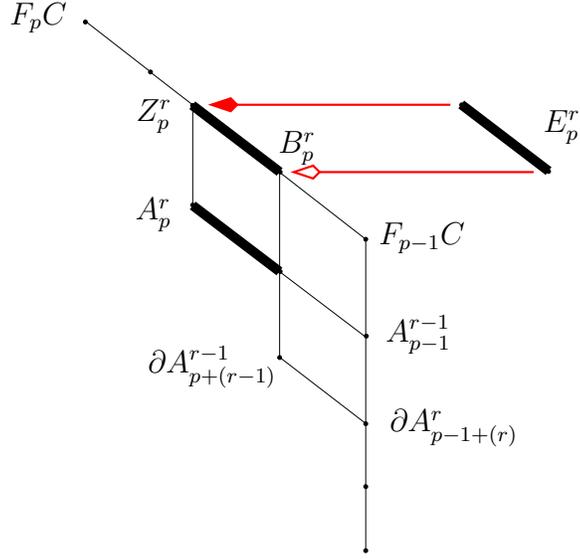}
  \end{minipage}
  \caption{The fundamental subobject lattice}
  \label{E_r}
\end{figure}

\bigskip
Setting $Z^\infty_p := \cap_{r=0}^\infty Z^r_p$ and $B^\infty_p := \cup_{r=0}^\infty B^r_p$ completes the 
\textbf{tower} of subobjects
\[
  F_{p-1} C =
  B^0_p \leq B^1_p \leq \cdots \leq B^r_p \leq \cdots \leq
  B^\infty_p \leq Z^\infty_p
  \leq \cdots \leq Z^r_p \leq \cdots \leq Z^1_p \leq Z^0_p
  = F_p C
\]
between $F_{p-1}C$ and $F_p C$.

From Figure~\ref{E_r} it is immediate that
\[
  E^r_p := \frac{Z^r_p}{B^r_p} \cong \frac{A^r_p}{\partial A^{r-1}_{p+(r-1)}+A^{r-1}_{p-1}}.
\]
It is now routine to verify that the total boundary operator $\partial$ induces morphisms
\[
  \partial^r_p:E^r_p \to E^r_{p-r}.
\]
And as mentioned in Section~\ref{les} these morphisms decrease the filtration degree by $r$. They complete the definition of the $r$-th sheet.

From the point of view of effective computations the above definition of $\partial^r_p$ \emph{is constructive}, as long as all involved objects are of \emph{finite type}. In fact, it can easily be turned into an algorithm using generalized maps. But since the filtered complexes relevant to our applications are total complexes of bicomplexes, the description of this algorithm is deferred to Section~\ref{bicomplexes}, where the bicomplex structure will be exploited.

To see that $(E^r)$ indeed defines a spectral sequence it remains to show the taking homology in $E^r$ reproduces the objects of $E^{r+1}$ up to (natural) isomorphisms. For this purpose one uses the statements encoded in Figure~\ref{E_r} to deduce that
\begin{enumerate}
  \item[(a)] ${Z^r_p}/{Z^{r+1}_p} \cong {B^{r+1}_{p-r}}/{B^r_{p-r}}$,
  \item[(b)] $\ker \partial^r_p \cong {Z^{r+1}_p}/{B^r_p}$,
  \item[(c)] $\img \partial^r_{p+r} \cong {B^{r+1}_p}/{B^r_p}$, and finally
  \item[(d)] $E^{r+1}_p \cong \ker \partial^r_p / \img \partial^r_{p+r}$.
\end{enumerate}
(c) follows from (a) and (b) since they state that $\partial^r_p$ decomposes as
\[
  E^r_p := {Z^r_p}/{B^r_p} \xrightarrow{\text{(b)}} {Z^r_p}/{Z^{r+1}_p}
 \xrightarrow{\text{(a)}} {B^{r+1}_{p-r}}/{B^r_{p-r}} \hookrightarrow {Z^r_{p-r}}/{B^r_{p-r}} =: E^r_{p-r},
\]
showing that $\img \partial^r_p \cong {B^{r+1}_p}/{B^r_p}$. Now replace $p$ by $p+r$. (d) is the first isomorphism theorem applied to $E^{r+1}_p:=Z^{r+1}_p/B^{r+1}_p$ using (b) and (c). For (a) and (b) see \cite[Lemma 5.4.7 and the subsequent discussion]{WeiHom}.

Before stating the main theorem we make some remarks about convergence. Recall that all our filtrations are assumed finite of length $m$. This means that $E^m$ runs out of arrows and thus stabilizes, i.e. $E^m = E^{m+1} = \cdots$. We already saw this for $m=2$ in Section~\ref{les}. As customary, the stable sheet is denoted by $E^\infty$. The stable form of Figure~\ref{E_r} is Figure~\ref{E_infty}, where $A_p^\infty:=\cup_{r=0}^\infty A^r_p$ and $A^\infty_{p+\infty}:=\cup_{r=0}^\infty A^r_{p+r}$.
\begin{figure}[htb]
  \begin{minipage}[c]{1.1\linewidth}
    \centering
    \psfrag{$F_p C$}{$F_p C$}
    \psfrag{$F_{p-1} C$}{$F_{p-1} C$}
    \psfrag{$E^r_p$}{$E^\infty_p$}
    \psfrag{$\iota^r_p$}{$\iota_p$}
    \psfrag{$Z^r_p$}{$Z^\infty_p$}
    \psfrag{$B^r_p$}{$B^\infty_p$}
    \psfrag{$A^r_p$}{$A^\infty_p$}
    \psfrag{$A^{r-1}_{p-1}$}{$A^\infty_{p-1}$}
    \psfrag{$partial(A^{r-1})$}{$\partial A^\infty_{p+\infty}$}
    \psfrag{$partial(A^r)$}{$\partial A^\infty_{p-1+\infty}$}
   \includegraphics[width=0.4\textwidth]{E_r.eps}
  \end{minipage}
  \caption{The stable fundamental subobject lattice}
  \label{E_infty}
\end{figure}

The identities
\begin{equation}\label{infty}
  A^\infty_p = \ker \partial_{\mid F_p C} = \{ c \in F_p C \mid \partial c = 0 \}
\end{equation}
and
\begin{equation}\label{p+infty}
  \partial A^\infty_{p+\infty} = \img \partial_{\mid F_p C} = \partial C \cap F_p C
\end{equation}
are direct consequences of the respective definitions.

\begin{axiom}[Beyond $E^\infty$]\label{mainthm}
  Let $C$ be a chain complex with an ascending $m$-step filtration. The generalized embedding $\iota:H(C) \to C$ divides all generalized embeddings  $\iota_p: E^\infty_p \to C$, called the \textbf{total embedding} of $E^\infty_p$. The quotients $\epsilon_p := \iota_p / \iota$ form an $m$-filtration system which computes the induced filtration on $H(C)$.
\end{axiom}
\begin{proof}
  We only need to verify the two lifting conditions for the pairs $(\iota,\iota_p)$. Everything else is immediate. For the morphism aid subobjects of $\iota_p$ and $\iota$ we have
\[
  L_{\iota_p}=\partial A^\infty_{p+\infty}+F_{p-1} C
\] (see Figure~\ref{E_infty}) and
\[
  L_\iota = \partial C.
\]
Define
\[
  L:=L_{\iota_p}+L_\iota=(\partial A^\infty_{p+\infty}+F_{p-1} C) + \partial C = \partial C + F_{p-1} C.
\]
\textbf{Condition (im)}: Since $\Img \iota_p = A^\infty_p + F_{p-1} C$ and $\Img \iota = \ker \partial$ we obtain
\begin{eqnarray*}
 \Img \widetilde{\iota}_p \leq \Img \widetilde{\iota} & \iff & (A^\infty_p + F_{p-1} C) + L \leq \ker\partial + L \\
 &\iff& A^\infty_p + \partial C+ F_{p-1} C \leq \ker\partial + F_{p-1} C.
\end{eqnarray*}
Now $\partial C \leq \ker \partial$ since $\partial$ is a boundary operator, and $A^\infty_p \leq \ker \partial$ by (\ref{infty}). \\
\textbf{Condition (eff)}:
\begin{eqnarray*}
  \Img\iota_p \cap L &=& (\partial C + F_{p-1} C) \cap (A^\infty_p + F_{p-1} C) \\
  & \stackrel{\text{(\ref{infty})}} {=} &(\partial C \cap F_p C) + F_{p-1} C \\
  & \stackrel{\text{(\ref{p+infty})}}{=} & \partial A^\infty_{p+\infty} + F_{p-1} C \\
  & = & L_{\iota_p}.
\end{eqnarray*}
The lifting lemma \ref{lifting_lemma} is now applicable, yielding the generalized embeddings $\epsilon_p := \iota_p / \iota$.
\end{proof}

Corollary \ref{coro_2-filt} is the special case $m=2$. In light of Remark \ref{comp} the theorem thus states that the induced filtration on the total (co)homology is effectively computable, as long as the generalized embeddings $\iota$ and $\iota_p$ are effectively computable for all $p$. Hence, it can be viewed as a (more) constructive version of the \textbf{classical convergence theorem} of spectral sequences of filtered complexes, a version that makes use of generalized embeddings:

\begin{axiom}[Classical convergence theorem {\cite[Thm.~5.5.1]{WeiHom}}]
  Let $C$ be chain complex with a finite filtration $(F_pC)$. Then the associated spectral sequence converges to $H_*(C)$:
  \[
    E^0_{pq} := F_p C_{p+q} / F_{p-1} C_{p+q} \Longrightarrow H_{p+q}(C).
  \]
\end{axiom}

Everything in this section can be reformulated for \emph{co}chain complexes and cohomological spectral sequences.

\section{Spectral sequences of bicomplexes}\label{bicomplexes}

Bicomplexes are one of the main sources for filtered complexes in algebra. They are less often encountered in topology. A \textbf{homological bicomplex} is a lattice $B=(B_{pq})$ ($p,q\in\Z$) of objects connected with \textbf{vertical} morphisms $\partial^\mathrm{v}$ pointing \emph{down} and \textbf{horizontal} morphisms $\partial^\mathrm{h}$ pointing \emph{left}, such that $\partial^\mathrm{v}\partial^\mathrm{h}+\partial^\mathrm{h}\partial^\mathrm{v}=0$.
\[
  \xymatrix{
    *=0{}
     \jumpdir{q}{/:a(90).5cm/}
    &
    B_{02}
    \ar@{.}[rd]
    \ar[d]^{\partial_\mathrm{v}}
    &
    B_{12}
    \ar@{.}[rd]
    \ar[d]^{\partial_\mathrm{v}}
    \ar[l]_{\partial_\mathrm{h}}
    &
    B_{22}
    \ar[d]^{\partial_\mathrm{v}}
    \ar[l]_{\partial_\mathrm{h}}
    \\
    &
    B_{01}
    \ar@{.}[rd]
    \ar[d]^{\partial_\mathrm{v}}
    &
    B_{11}
    \ar@{.}[rd]
    \ar[d]^{\partial_\mathrm{v}}
    \ar[l]_{\partial_\mathrm{h}}
    &
    B_{21}
    \ar[d]^{\partial_\mathrm{v}}
    \ar[l]_{\partial_\mathrm{h}}
    \\
    &
    B_{00}
    &
    B_{10}
    \ar[l]_{\partial_\mathrm{h}}
    &
    B_{20}
    \ar[l]_{\partial_\mathrm{h}}
    \\
    *=0{} \ar[uuu] \ar[rrr]
    &&&
    *=0{}
    \jumpdir{p}{/:a(-100).5cm/}
  }
\]

The \textbf{sign trick} $\hat{\partial}_{pq}:=(-1)^p \partial^\mathrm{v}_{pq}$ converts the \textbf{anticommutative} squares into commutative ones, and hence turns the bicomplex into a \textbf{complex of complexes} connected with chain maps as morphisms, and vice versa.

The direct sum of objects $\Tot(B)_n := \bigoplus_{p+q=n} B_{pq}$ together with the \textbf{total boundary operator} $\partial_n := \sum_{p+q=n} \partial^\mathrm{v}_{pq}+\partial^\mathrm{h}_{pq}$ form a chain complex called the \textbf{the total complex} associated to the bicomplex $B$. $\partial \partial=0$ is a direct consequence of the anticommutativity.

The vertical morphisms $d_\mathrm{v}$ of a \textbf{cohomological bicomplex} $(B^{pq})$ point \emph{up} and the horizontal $d_\mathrm{h}$ point \emph{right}. We assume all bicomplexes bounded, i.e.\  only finitely many objects $B_{pq}$ are different from zero.

There exists a natural so-called \textbf{column filtration} of the total complex $\Tot(B)$ such that the $0$-th page $E^0=(E^0_{pq}) = (B_{pq})$ of the spectral sequence associated to this filtration consists of the vertical arrows of $B$ and the $1$-st page $E^1$ contains morphisms induced by the vertical ones. Its associated spectral sequence is called the \textbf{first spectral sequence} of the bicomplex $B$ and is often denoted by ${}^\mathrm{I}E$. For a formal definition see \cite[Def.~5.6.1]{WeiHom}. The \textbf{second spectral sequence} is the (first) spectral sequence of the \textbf{transposed bicomplex} $^\mathrm{tr} B = (^\mathrm{tr} B_{pq}) := (B_{qp})$. It is denoted by ${}^\mathrm{II}E$. Note that $\Tot(B)=\Tot(^\mathrm{tr}B)$, only the two corresponding filtrations and their induced filtrations on the total cohomology $H_*(\Tot(B))$ differ in general. So the short notation
\[
  {}^\mathrm{I}E^a_{pq} \Longrightarrow H_{p+q}(\Tot(B)) \Longleftarrow  {}^\mathrm{II}E^a_{pq}
\]
refers in general to two different filtrations of $H_{p+q}(\Tot(B))$.

Here is an algorithm using generalized maps to compute the arrows
\[
  \partial^r_{pq}:E^r_{pq} \to E^r_{p-r,q+r-1}
\]
of the $r$-th term of the homological (first) spectral sequence $E^r$. Again, everything can be easily adapted for the cohomological case. Denote by
\[
  \alpha_S: E^r_{pq} \to B_{pq} \quad \mbox{resp.} \quad  \alpha_T: E^r_{p-r,q+r-1} \to B_{p-r,q+r-1}
\]
the generalized embedding of the \emph{s}ource resp.\  \emph{t}arget of $\partial^r_{pq}$ into the object $B_{pq}=E^0_{pq} \leq \Tot(B)_{p+q}$ resp.\  $B_{p-r,q+r-1}\leq\Tot(B)_{p+q-1}$. These so-called \textbf{absolute embeddings} are the successive compositions of the \textbf{relative embeddings} $E^r_{pq} \to E^{r-1}_{pq}$. For the sake of completeness we also mention the \textbf{total embeddings}
\[
  \iota_S: E^r_{pq} \to \Tot(B)_{p+q} \quad \mbox{resp.} \quad  \iota_T: E^r_{p-r,q+r-1} \to \Tot(B)_{p+q-1},
\]
the compositions of $\alpha_S$ resp.\  $\alpha_T$ with the \emph{generalized} embeddings\footnote{It identifies $B_{pq}$ with the \emph{subfactor} of $\Tot(B)_{p+q}$ dictated by the filtration.} $B_{pq} \to \Tot(B)_{p+q}$ resp.\  $B_{p-r,q+r-1} \to \Tot(B)_{p+q-1}$.
\begin{figure}[htb]
  \begin{minipage}[c]{1\linewidth}
    \centering
    \psfrag{$E^infty_{pq}$}{$E^\infty_{pq}$}
    \psfrag{$E^r_{pq}$}{$E^r_{pq}$}
    \psfrag{$E^0_{pq}$}{$E^0_{pq}$}
    \psfrag{$C_{p+q}$}{$C_{p+q}=\Tot(B)_{p+q}$}
    \psfrag{$alpha$}{$\alpha_{pq}$}
    \psfrag{$iota$}{$\iota_{pq}$}
    \psfrag{$...$}{$\cdots$}
    \includegraphics[width=0.6\textwidth]{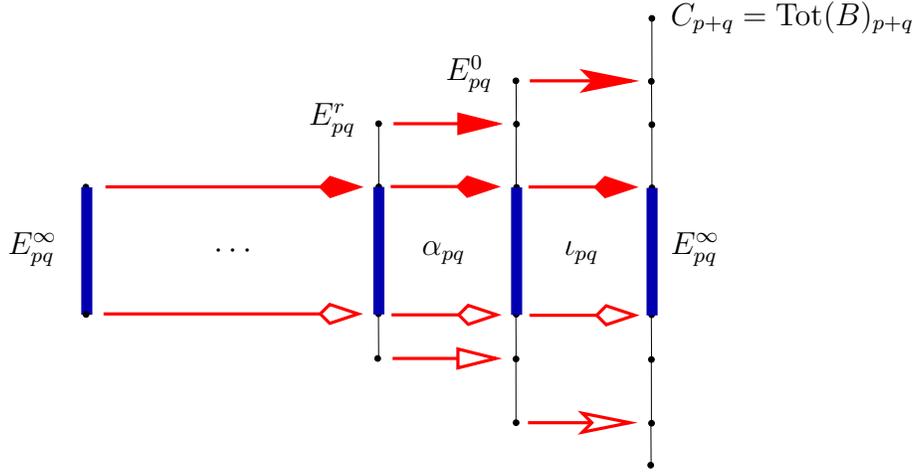}
  \end{minipage}
  \caption{The relative, absolute, and total embeddings}
  \label{TotEmb}
\end{figure}

For $r>1$ let
\[ 
  h^r_{pq}: {\color{brown} B_{pq}} \to {\color{blue} \bigoplus_{i=1}^{r-1} B_{p-i,q+i-1}}
  \quad \mbox{and} \quad
  v^r_{p-r+1,q+r-1}: {\color{red} B_{p-r+1,q+r-1}} \to {\color{blue} \bigoplus_{i=1}^{r-1} B_{p-i,q+i-1}}
\]
be the restrictions of the total boundary operator $\partial_{p+q}$ to the specified sources and targets. Similarly, for $r>2$ let
\[
  l^r_{pq}:{\color{magenta} \bigoplus_{i=1}^{r-2} B_{p-i,q+i}} \to {\color{blue}\bigoplus_{i=1}^{r-1} B_{p-i,q+i-1}},
\]
again the restriction of the total boundary operator $\partial_{p+q}$ to the specified source and target.
\[
  \xymatrix@!C=5.5pc{
    E^r_{p-r,q+r-1} \ar@{_{(}->}[d]_{\alpha_T} \\
    {\color{darkgreen} B_{p-r,q+r-1}} & {\color{red} B_{p-r+1,q+r-1}} \ar[l]_{\partial^\mathrm{h}} \ar[d]^{\partial^\mathrm{v}}\\
    & {\color{blue} B_{p-r+1,q+r-2}} \ar@{.}[rd] & {\color{magenta} B_{p-r+2,q+r-2}} \ar[l] \ar@{.>}[d] \ar@{--}[rd] \\
    & & {\color{blue} \ddots} \ar@{.}[rd] & {\color{magenta} B_{p-1,q+1}} \ar@{.>}[l] \ar[d] \\
    & & & {\color{blue} B_{p-1,q}} & \ar[l]_<(.25){\partial^\mathrm{h}} {\color{brown} B_{pq}} \\
    & & & & E_{pq} \ar@{^{(}->}[u]^{\alpha_S}
  }
\]

We distinguish four cases $r=0,1,2$, and $r>2$.
\begin{enumerate}
  \item[$r=0$:] $\partial^0_{pq}:=\partial^\mathrm{v}_{pq}$. Note that $E^0_{pq}:=B_{pq}$.
  \item[$r=1$:] $\partial^1_{pq}:=\alpha_T ^{-1} \circ (\partial^\mathrm{h}_{pq}\circ\alpha_S)$.
  \item[$r=2$:] $\partial^2_{pq}:=\alpha_T ^{-1} \circ (\partial^\mathrm{h}_{p-1,q+1} \circ (-\beta^{-1} \circ (h^2_{pq}\circ\alpha_S)))$, where $\beta:=v^2_{p-1,q+1}$. Note that $h^2_{pq} = \partial^\mathrm{h}_{pq}$ and $v^2_{p-1,q+1}=\partial^\mathrm{v}_{p-1,q+1}$.
  \item[$r>2$:] $\partial^r_{pq}:=\alpha_T^{-1} \circ (\partial^\mathrm{h}_{p-r+1,q+r-1} \circ (-\beta^{-1} \circ (h^r_{pq}\circ\alpha_S)))$, with $\beta:=(v^r_{p-r+1,q+r-1},l^r_{pq})$, the coar\-sening of $v^r_{p-r+1,q+r-1}$ with aid $l_{pq}$. We say: $v^r_{p-r+1,q+r-1}$ aided by $l^r_{pq}$ lifts $h^r_{pq}\circ\alpha_S$.
\end{enumerate}

\smallskip
We announced an algorithm and provided closed formulas. This is the true value of generalized maps mentioned in the Introduction. As an easy exercise, the reader might try to rephrase the diagram chasing of the snake lemma as a closed formula in terms of generalized maps. The concept of a generalized map evolved during the implementation of the {\tt homalg} package in {\sf GAP} \cite{homalg-package}.

It follows from Remark~\ref{comp} that the spectral sequence of a finite type bounded bicomplex (in fact, of a finite type complex with finite filtration) over a computable ring is effectively computable (cf.~Def.~\ref{compdef}).  The {\tt homalg} package \cite{homalg-package} contains routines to compute spectral sequences of bicomplexes.

\medskip
We end this section with a simple example from linear algebra. Let $k$ be a field and $\lambda\in k$ a field element. The {\sc Jordan}-form matrix
\[
  J(\lambda) =
      \left(\begin{matrix}
      \lambda & 1 & \cdot \\
      \cdot & \lambda & 1 \\
      \cdot & \cdot & \lambda
    \end{matrix}\right) \in k^{3\times 3}
\]
turns $V := k^{1\times 3}$ into a left $k[x]$-module (of finite length), where $x$ acts via $J(\lambda)$, i.e.\  $x v := J(\lambda) v$ for all $v\in V$. The $k[x]$-module $V$ is filtered and the filtrations stems from a bicomplex:

\begin{exmp}[\textbf{Spectrum} of an endomorphism]
  Let $k$ be a field and $\lambda\in k$. Consider the second quadrant bicomplex $B_\lambda$
  \[
    \xymatrix{
      B_{-2,3} \ar[d]_{\left( \begin{smallmatrix} x-\lambda \end{smallmatrix} \right)} \\
      B_{-2,2} & B_{-1,2} \ar[l]_{\left( \begin{smallmatrix} -1 \end{smallmatrix} \right)}
                        \ar[d]_{-\left( \begin{smallmatrix} x-\lambda \end{smallmatrix} \right)} \\
      & B_{-1,1} & B_{0,1} \ar[l]_{\left( \begin{smallmatrix} -1 \end{smallmatrix} \right)}
                            \ar[d]_{\left( \begin{smallmatrix} x-\lambda \end{smallmatrix} \right)} \\
      & & B_{0,0}
    }
  \]
  with $B_{0,0}=B_{0,1}=B_{-1,1}=B_{-1,2}=B_{-2,2}=B_{-2,3}=k[x]$, all other spots being zero. The total complex contains exactly two nontrivial $k[x]$-modules at degrees $0$ and $1$ and a single nontrivial morphism
  \[
    \partial_1(\lambda) : \xymatrix@1{ \Tot(B)_1=k[x]^{1\times 3} \ar[r] & k[x]^{1\times 3} = \Tot(B)_0}
  \]
  with matrix
  \[
    x\Id - J(\lambda) =
    \left(\begin{matrix}
      x-\lambda & -1 & \cdot \\
      \cdot & x-\lambda & -1 \\
      \cdot & \cdot & x-\lambda
    \end{matrix}\right).
  \]
  The first spectral sequences ${}^\mathrm{I}E$ lives in the second quadrant and stabilizes already at ${}^\mathrm{I}E^1 =: {}^\mathrm{I}E^\infty$
  \[
    \xymatrix@!=0.8pc{
      \cdot\ar@{.}[rd] & \cdot \ar@{.}[rd] & \cdot \\
      {}^\mathrm{I}E^1_{-2,-2} \ar@{.}[rd] & \cdot \ar@{.}[rd] & \cdot\\
      \cdot \ar@{.}[rd] & {}^\mathrm{I}E^1_{-1,-1} \ar@{.}[rd] & \cdot \\
      \cdot & \cdot & {}^\mathrm{I}E^1_{0,0}
    }
  \]
  with ${}^\mathrm{I}E^\infty_{0,0}={}^\mathrm{I}E^\infty_{-1,-1}={}^\mathrm{I}E^\infty_{-2,-2}=k[x]/\langle x - \lambda \rangle$.
  
  The second spectral sequences ${}^\mathrm{II}E$ lives in the fourth quadrant, has only zero arrows at levels $1$ and $2$
  \[
    \xymatrix@!=0.9pc{
      {}^\mathrm{II}E^1_{0,0} \ar@{.}[rd] & \cdot \ar@{.}[rd] & \cdot \ar@{.}[rd] & \cdot \\
      \cdot \ar@{.}[rd] & \cdot \ar@{.}[rd] & \cdot \ar@{.}[rd] & \cdot \\
      \cdot & \cdot & \cdot & {}^\mathrm{II}E^1_{3,-2}
    } \quad\quad
    \xymatrix@!=0.8pc{
      {}^\mathrm{II}E^2_{0,0} \ar@{.}[rd] & \cdot \ar@{.}[rd] & \cdot \ar@{.}[rd] & \cdot \\
      \cdot \ar@{.}[rd] & \cdot \ar@{.}[rd] & \cdot \ar@{.}[rd] & \cdot \\
      \cdot & \cdot & \cdot & {}^\mathrm{II}E^2_{3,-2}
    }
  \]
  with ${}^\mathrm{II}E^1_{0,0}={}^\mathrm{II}E^1_{3,-2}=k[x]$, and hence ${}^\mathrm{II}E^2_{0,0}={}^\mathrm{II}E^2_{3,-2}=k[x]={}^\mathrm{II}E^3_{0,0}={}^\mathrm{II}E^3_{3,-2}$. At level $3$ there exists a single nonzero arrow $\partial^3_{3,-2}$ with matrix $\left(\begin{matrix}(x-\lambda)^3\end{matrix}\right)$:
  \[
    \xymatrix@!=0.9pc{
      {}^\mathrm{II}E^3_{0,0} \ar@{.}[rd] & \cdot \ar@{.}[rd] & \cdot \ar@{.}[rd] & \cdot \\
      \cdot \ar@{.}[rd] & \cdot \ar@{.}[rd] & \cdot \ar@{.}[rd] & \cdot\\
      \cdot & \cdot & \cdot & {}^\mathrm{II}E^3_{3,-2} \ar[uulll] |{\partial^3_{3,-2}}
    }
  \]
  ${}^\mathrm{II}E$ finally collapses to its $p$-axes at  ${}^\mathrm{II}E^4 =: {}^\mathrm{II}E^\infty$
  \[
    \xymatrix@!=0.9pc{
      {}^\mathrm{II}E^4_{0,0} \ar@{.}[rd] & \cdot \ar@{.}[rd] & \cdot \ar@{.}[rd] & \cdot \\
      \cdot \ar@{.}[rd] & \cdot \ar@{.}[rd] & \cdot \ar@{.}[rd] & \cdot \\
      \cdot & \cdot & \cdot & \cdot
    }
  \]
  with ${}^\mathrm{II}E^\infty_{0,0}  =  k[x] / \langle ( x - \lambda )^3\rangle$, providing a spectral sequence proof for the elementary fact that
  \[
    \coker\partial_1(\lambda) \cong k[x]/\langle ( x - \lambda )^3\rangle.
  \]
Conversely, this isomorphism implies that the matrix of the morphism $\partial^3_{3,-2}$ is equal to $\left(\begin{matrix}( x - \lambda )^3\end{matrix}\right)$, up to a unit $a\in k^\times$.
\end{exmp}

\section{The {\sc Cartan-Eilenberg} resolution of a complex}

The \textbf{{\sc Cartan-Eilenberg} resolution} generalizes the \textbf{horse shoe lemma} in the following sense: The horse shoe lemma produces a \textbf{simultaneous} projective resolution\footnote{We will only refer to projective resolutions as they are more relevant to effective computations.}
\[
  \xymatrix{
    & 0 \ar[d] & 0 \ar[d]  & \cdots & 0 \ar[d] \\
    0 & M' \ar[l] \ar[d] & P'_0 \ar[l] \ar[d] & \cdots \ar[l] & P'_d \ar[l] \ar[d] & 0 \ar[l] \\
    0 & M \ar[l] \ar[d] & P_0 \ar[l] \ar[d] & \cdots \ar[l] & P_d \ar[l] \ar[d] & 0 \ar[l] \\
    0 & M'' \ar[l] \ar[d] & P''_0 \ar[l] \ar[d] & \cdots \ar[l] & P''_d \ar[l] \ar[d] & 0 \ar[l] \\
    & 0 & 0 & \cdots & 0
  }
\]
of a short exact sequence $0 \xleftarrow{} M'' \xleftarrow{} M \xleftarrow{} M' \xleftarrow{} 0$, where simultaneous means that each row is a projective resolution and all columns are exact. Now let us look at this threefold resolution in the following way: The short exact sequence defines a $2$-step filtration of the object $M$ with graded parts $M'$ and $M''$ and the horse shoe lemma states that any resolutions of the graded parts can be put together to a resolution of the total object $M$. In fact, as $P_i''$ is projective, it follows that the total object $P_i$ must even be the direct sum of the graded parts $P_i'$ and $P_i''$. The non-triviality of the filtration on $M$ is reflected in the fact that the morphisms of the total resolution $P_*$ are in general not merely the direct sum of the morphisms in the resolutions $P_*'$ and $P_*''$ of the graded parts $M'$ and $M''$. This statement can now be generalized to $m$-step filtrations simply by applying the ($2$-step) horse shoe lemma inductively.

Now consider a complex $(C,\partial)$, which is not necessarily exact. On each object $C_n$ the complex structure induces a $3$-step filtration $0\leq B_n \leq Z_n \leq C_n$, with boundaries $B_n := \img \partial_{n+1}$ and cycles $Z_n:=\ker \partial_n$. The above discussion now applies to the three graded parts $B_n$, $H_n:=Z_n/B_n$ and $C_n/Z_n$ and any three resolution thereof can be put together to a resolution of the total object $C_n$. If one takes into account the fact that $\partial_{n+1}$ induces an isomorphism between $C_{n+1}/Z_{n+1}$ and $B_n$ (for all $n$, by the homomorphism theorem), then all total resolutions of all the $C_n$'s can be constructed in a compatible way so that they fit together in one complex of complexes. This complex is called the {\sc Cartan-Eilenberg} resolution of the complex $C$.

A formal version of the above discussion can be found in \cite[Lemma~9.4]{HS} or \cite[Lemma~5.7.2]{WeiHom}. Since the projective horse shoe lemma is constructive, the projective {\sc Cartan-Eilenberg} resolution is so as well.

\section{{\sc Grothendieck}'s spectral sequences}

Let $\mathcal{C} \xleftarrow{F} \mathcal{B} \xleftarrow{G} \mathcal{A}$ be composable functors of abelian categories. The so-called {\sc Gro\-then\-dieck} spectral sequence relates, under mild assumptions, the composition of the derivations of $F$ and $G$ with the derivation of their composition $F\circ G$. There are 16 versions of the {\sc Grothendieck} spectral sequence, depending on whether $F$ resp.\  $G$ is co- or contravariant, and whether $F$ resp.\  $G$ is being left or right derived. Four of them do not use injective resolutions and are therefore rather directly accessible to a computer. In this section two versions out of the four are reviewed: The filtrations of $L\otimes_D M$ and $\Hom_D(M,N)$ mentioned in the Introduction are recovered in the next section as the spectral filtrations induced by these two {\sc Grothen\-dieck} spectral sequences, after appropriately choosing the functors $F$ and $G$.

\begin{axiom}[{\sc Grothendieck} spectral sequence, {\cite[Thm.~11.41]{Rot}}]\label{Gr1}
Let $F$ and $G$ be contravariant functors and let every object in $\mathcal{A}$ and $\mathcal{B}$ has a \emph{finite} projective resolution. Under the assumptions that
\begin{enumerate}
  \item $G$ maps projective objects to $F$-acyclic objects and that
  \item $F$ is left exact,
\end{enumerate}
then there exists a second quadrant homological spectral sequence with
\[
  E^2_{pq}=\RR^{-p}F \circ \RR^qG \Longrightarrow \LL_{p+q}(F\circ G).
\]
\end{axiom}
\begin{proof}
Let $M$ be an object in $\mathcal{A}$ and $P_\bullet=(P_p)$ a finite projective resolution of $M$. Denote by $CE=(CE^{p,q})$ the projective {\sc Cartan-Eilenberg} resolution of the cocomplex $(Q^p):=(G(P_p))$. It exists since $\mathcal{B}$ has enough projectives. Note that $q\leq 0$ since $CE$ is a cohomological bicomplex. Define the homological bicomplex $B=(B_{p,q}):=\left(F(CE^{p,q})\right)$. We call $B$ the \textbf{Grothendieck bicomplex} associated to $M$, $F$, and $G$. It lives in the fourth quadrant and is bounded in both directions. \\
\textbf{The first spectral sequence $^\mathrm{I}E$}: \\
For fixed $p$ the vertical cocomplex $CE^{p,\bullet}$ is, by construction, a projective resolution of $G(P_p)$. Hence $^\mathrm{I}E^1_{pq}=\RR^{-q}F(G(P_p))$. But since $G(P_p)$ is $F$-acyclic by assumption (1), the first sheet collapses to the $0$-th row. The left exactness of $F$ implies that  $R^0F=F$ and hence $^\mathrm{I}E^1_{p0}=(F\circ G)(P_p)$. I.e.\  the $0$-th row of $^\mathrm{I}E^1$ is nothing but the covariant functor $F\circ G$ applied to the projective resolution $(P_p)$ of $M$. The first spectral sequences of $B$ thus stabilizes at level $2$ with the single row $^\mathrm{I}E^2_{n,0}=\LL_n(F\circ G)(M)$. \\
\textbf{The second spectral sequence $^\mathrm{II}E$}: \\
The second spectral sequence of the bicomplex $B$ is by definition the spectral sequence of its transposed $(^\mathrm{tr}B_{pq}):=(B_{qp})$, a second quadrant bicomplex. Obviously $^\mathrm{tr}B=F(^\mathrm{tr}CE)$. By definition, the $q$-th row ${}^\mathrm{II}E^1_{\bullet,q}:=H^\mathrm{vert}_{\bullet,q}(^\mathrm{tr}B)=H^\mathrm{vert}_{\bullet,q}(F(^\mathrm{tr}CE))=F(H_\mathrm{vert}^{\bullet,q}(^\mathrm{tr}CE))$, where the last equality follows from the properties of the {\sc Cartan-Eilenberg} resolution and the additivity of $F$. Now recall that the vertical cohomologies $H_\mathrm{vert}^{\bullet,q}(^\mathrm{tr}CE)$ are for fixed $q$, again by construction, projective resolutions of the cohomology $H^q(G(P_\bullet))=:\RR^q G(M)$. Hence $^\mathrm{II}E^2_{pq} = \RR^{-p}F (\RR^q G(M))$.
\end{proof}

The proof shows that assumptions (1) and (2) only involve the first spectral sequence. Assumption (1) guaranteed the collapse of the first spectral sequence at the first level, while (2) ensures that the natural transformation $F \to \RR^0 F$ is an equivalence. In other words, dropping (2) means replacing $\LL_{p+q}(F\circ G)$ by $\LL_{p+q}(\RR^0F\circ G)$.

\begin{axiom}[{\sc Grothendieck} spectral sequence]\label{Gr2}
Let $F$ be a covariant and $G$ a contravariant functor and let every object in $\mathcal{A}$ and $\mathcal{B}$ has a \emph{finite} projective resolution. Under the assumptions that
\begin{enumerate}
  \item $G$ maps projective objects to $F$-acyclic objects and that
  \item $F$ is right exact,
\end{enumerate}
then there exists a second quadrant cohomological spectral sequence with
\[
  E^2_{pq}=\LL_{-p}F \circ \RR^qG \Longrightarrow \RR^{p+q}(F\circ G).
\]
\end{axiom}
\begin{proof}
Again the first spectral sequence is a fourth quadrant spectral sequence while the second lives in the second quadrant. Assumption (2) ensures that the natural transformation $\LL^0 F \to F$ is an equivalence. The above proof and the subsequent remark can be copied with the obvious modifications.
\end{proof}

\begin{rmrk}[One sided boundedness]\label{bounded}
The existence of finite projective resolutions in $\mathcal{A}$ and $\mathcal{B}$ led the spectral sequences to be bounded in both directions. In order to avoid convergence subtleties it would suffice to assume boundedness in just one direction by requiring that either $\mathcal{A}$ or $\mathcal{B}$ allows finite projective resolutions while the other has enough projectives. The assumption of the existence of \emph{finite} projective resp.\  injective resolutions can be dropped when dealing with the versions of the {\sc Grothendieck} spectral sequences that live in the first resp.\  third quadrant.
\end{rmrk}

\section{Applications}\label{appl}

This section recalls how the natural filtrations mentioned in examples (a), (a'), and (d) of the Introduction can be recovered as \textbf{spectral filtrations}.

Theorems~\ref{Gr1} and \ref{Gr2} admit an obvious generalization. The composed functor $F\circ G$ can be replaced by a functor $H$ that coincides with $F\circ G$ on projectives (for other versions of the {\sc Grothendieck} spectral sequence the ``projectives'' has to be replaced by ``injectives''). As usual, $D$ is an associative ring with $1$. $\Ext^n_D$ and $\Tor_n^D$ are abbreviated as $\Ext^n$ and $\Tor_n$.

\bigskip
\noindent
\textbf{Assumption:} In this section the left \emph{or} right global dimension\footnote{
Recall, the \textbf{left} \textbf{global (homological) dimension} is the supremum over all projective dimensions of \emph{left} $D$-modules (see Subsection~\ref{codegree}). If $D$ is left {\sc Noether}ian, then the left global dimension of $D$ coincides with the \textbf{weak global (homological) dimension}, which is the largest integer $\mu$ such that $\Tor^D_\mu(M,N) \neq 0$ for some right module $M$ and left module $N$, otherwise infinity (cf.~\cite[7.1.9]{MR}). This last definition is obviously left-right symmetric. The same is valid if ``left'' is replaced by ``right''.} of $D$ is assumed finite. The involved spectral sequences will then be bounded in (at least) one direction (see Remark~\ref{bounded}).

\subsection{\texorpdfstring{The double-$\Ext$ spectral sequence and the filtration of $\Tor$}{The double-Ext spectral sequence and the filtration of Tor}}


\begin{coro}[The double-$\Ext$ spectral sequence]\label{doubleExt}
  Let $M$ be a left $D$-module and $L$ a right $D$-module. Then there exists a second quadrant homological spectral sequence with
  \[
    E^2_{pq}=\Ext^{-p}(\Ext^q(M,D),L) \Longrightarrow \Tor_{p+q}(L,M).
  \]
  In particular, there exists an ascending filtration of $\Tor_{p+q}(L,M)$ with $\gr_p \Tor_{p+q}(L,M)$ naturally isomorphic to a subfactor of $\Ext^{-p}(\Ext^q(M,D),L)$, $p\leq 0$.
\end{coro}


The special case $p+q=0$ recovers the filtration of $L\otimes M$ mentioned in Example (a) of the Introduction via the natural isomorphism $L\otimes M \cong \Tor_0 (L,M)$.

\subsubsection{Using the {\sc Grothendieck} bicomplex}\label{HGB}

Corollary~\ref{doubleExt} is a consequence of Theorem~\ref{Gr1} for $F:=\Hom_D(-,L)$ and $G:=\Hom_D(-,D)$, since $F\circ G$ coincides with $L\otimes_D -$ on projectives.

To be able to effectively compute double-$\Ext$ (groups in) the {\sc Grothendieck} bicomplex the ring $D$ must be computable in the sense that \emph{two} sided inhomogeneous linear systems $A_1 X_1 + X_2 A_2 = B$ must be effectively solvable, where $A_1$, $A_2$, and $B$ are matrices over $D$ (see~\cite[Subsection~6.2.4]{BR}). This is immediate for computable commutative rings (cf.~Def.~\ref{compdef}). In \ref{ExtExt} an example over a commutative ring is treated.

\subsubsection{Using the bicomplex $I_L\otimes P^M$}

The \textbf{bifunctoriality} of $\otimes$ leads to the following homological bicomplex
\[
  B := I_L\otimes P^M \cong \Hom(\Hom(P^M,D),I_L),
\]
where $P^M$ is an injective resolution of $M$ and $I_L$ is an injective resolution of $L$. Starting from $r=2$ the first and second spectral sequence of $B$ coincide with those of the {\sc Grothendieck} bicomplex associated to $M$, $F:=\Hom_D(-,L)$, and $G:=\Hom_D(-,D)$. In contrast to the {\sc Grothendieck} bicomplex the bicomplex $B$ is over most of the interesting rings in general highly nonconstructive as an injective resolution enters its definition. In \cite[Lemma~1.1.8]{HL} a sheaf variant of this bicomplex was used to ``compute'' the purity filtration (see below).

\subsubsection{The bidualizing complex}\label{bidual}

Taking $L=D$ as a right $D$-module in Corollary~\ref{doubleExt} recovers the \textbf{bidualizing spectral sequence} of {\sc J.-E.~Roos} \cite{Roos}.
\[
  E^2_{pq}=\Ext^{-p}(\Ext^q(M,D),D) \Longrightarrow \left\{\begin{array}{cc} M & \mbox{ for } p+q = 0, \\ 0 & \mbox{ otherwise.} \end{array}\right.
\]
The {\sc Grothendieck} bicomplex is then known as the \textbf{bidualizing complex}. The case $p+q=0$ defines the \textbf{purity filtration}\footnote{Unlike \cite[Chap.~2, Subsection~4.15]{Bjo}, we only make the weaker assumption stated at the beginning of the section.} $(\tor_{-c} M)$ of $M$, which was motivated in Example (a') of the Introduction. For more details cf.~\cite[Chap.~2, §5,7]{Bjo}.

The module $M_c = E^\infty_{-c,c}$ is for $c=0$ and $c=1$ a submodule of $\Ext^c(\Ext^c(M,D),D)=E^2_{-c,c}$ and for $c\geq 2$ in general only a subfactor. All this is obvious from the shape of the bidualizing spectral sequence.

Since $M_c = \tor_{-c} M / \tor_{-(c+1)}M$ it follows that the \textbf{higher evaluations maps} $\varepsilon_c$
\[
  0 \xrightarrow{} \tor_{-(c+1)}M \xrightarrow{} \tor_{-c} M \xrightarrow{\varepsilon_c} \Ext^c_D(\Ext^c_D(M,D),D)
\]
 mentioned in the Introduction are only a different way of writing the generalized embeddings
\[
  \bar{\varepsilon}_c:M_c \to \Ext^c(\Ext^c(M,D),D).
\]
So without further assumptions $\varepsilon_c$ (resp.\  $\bar{\varepsilon}_c$) is known to be an ordinary morphism (resp.\  embedding) only for $c=0$ and $c=1$. Now assuming that $E^2_{pq}:=\Ext^{-p}(\Ext^q(M,D),D)$ vanishes\footnote{This condition is satisfied for an \textbf{{\sc Auslander} regular} ring $D$: $\Ext^{-p}(\Ext^q(M,D),D) = 0$ for all $p+q>0$ and all $D$-modules $M$. See \cite[Chap.~2: Cor.~5.18, Cor.~7.5]{Bjo}.} for $p+q=1$, then all arrows ending at total degree $p+q=0$ vanish (as they all start at total degree $p+q=1$). It follows that for all $c$ the module $M_c$ is not merely a subfactor of $\Ext^c(\Ext^c(M,D),D)$ but a submodule, or, equivalently, $\varepsilon_c$ (resp.\  $\bar{\varepsilon}_c$) is an ordinary morphism (resp.\  embedding).

In any case the module $\Ext^c(\Ext^c(M,D),D)$ is called the \textbf{reflexive hull} of the \textbf{pure} subfactor $M_c$.

\begin{defn}[Pure, reflexively pure]
 A module $M$ is called \textbf{pure} if it consists of exactly one nontrivial pure subfactor $M_c$ or is zero. A nontrivial module $M$ is called \textbf{reflexively pure} if it is pure and if the generalized embedding $M=M_c \to \Ext^c(\Ext^c(M,D),D)$ is an isomorphism. Define the zero module to be reflexively pure.
\end{defn}

If $M$ is a finitely generated $D$-module, then all ingredients of the bidualizing complex are again finitely generated (projective) $D$-modules, even if the ring $D$ is \emph{non}commutative. It follows that the purity filtration over a computable ring $D$ is effectively computable. A commutative and a noncommutative example are given in \ref{PurityFiltration} and \ref{PurityFiltration:A3} respectively. The latter demonstrates how the purity filtration (as a filtration that always exists) can be used to transform a linear system of PDEs into a triangular form where now a cascade integration strategy can be used to obtain exact solutions. The idea of viewing a linear system of PDEs as a module over an appropriate ring of differential operators was emphasized by {\sc B.~Malgrange} in the late 1960's and according to him goes back to {\sc Emmy Noether}.

\subsubsection{Criterions for reflexive purity}\label{reflexive}

This subsection lists some simple criterions for reflexive purity of modules.

First note that the existence of the bidualizing spectral sequence immediately implies that the set $c(M):=\{c \geq 0 \mid \Ext^c_D(M,D)\neq 0\}$ is empty only if $M=0$. Recall that if $c(M)$ is nonempty, then its minimum is called the \textbf{grade} or \textbf{codimension} of $M$ and denoted by $j(M)$ or $\codim M$. The codimension of the zero module is set to be $\infty$. Further define $\bar{q}(M):=\sup c(M)$ in case $c(M)\neq \emptyset$, and $\infty$ otherwise.

All of the following arguments make use of the shape of the bidualizing spectral sequence in the respective situation.

$\bullet$ If $c(M)$ contains a single element, i.e.\  if $\codim M = \bar{q}(M) =: \bar{q} < \infty$, then $M=M_{\bar{q}}$ is reflexively pure of codimension $\bar{q}$, giving a simple spectral sequence proof of \cite[Thm.~7]{QEB}.

For the remaining criterions assume that $\Ext^{-p}(\Ext^q(M,D),D) = 0$ for $p+q=1$:

$\bullet$ If $\bar{q}:=\bar{q}(M)$ is finite, then $E^2_{-\bar{q},\bar{q}}=E^\infty_{-\bar{q},\bar{q}}$, i.e.\  $M_{\bar{q}}$ is reflexively pure (possibly zero). This generalizes the above criterion (under the assumption just made).

$\bullet$ Now if $M$ is a left (resp.\  right) $D$-module, then assume further that the right (resp.\  left) global dimension $d$ of the ring $D$ is finite. It follows that $E^2_{-c,c}=E^\infty_{-c,c}$ for $c=d$ and $c=d-1$. This means that under the above assumptions the subfactors $M_d$ and $M_{d-1}$ are always reflexively pure\footnote{In case $D=A_n$, the $n$-th {\sc Weyl} algebra over a field, this says that \textbf{holonomic} and \textbf{subholonomic} modules are reflexively pure. See~\cite[Chap.~2, §7]{Bjo}.}.

\subsubsection{Codegree of purity}\label{codegree}

As a {\sc Grothendieck} spectral sequence the bidualizing spectral sequence becomes intrinsic at level $2$. Each $E^2_{-c,c}$ starts to ``shrink'' until it stabilizes at $E^\infty_{-c,c}=M_c$. Motivated by this define the \textbf{codegree of purity} $\cp M$ of a module $M$ as follows:  Set $\cp M$ to $\infty$ if $M$ is not pure. Otherwise $\cp M$ is a tuple of nonnegative integers, the length of which is one plus the number of times $E^a_{-c,c}$ shrinks (nontrivially\footnote{i.e.\  passes to a \emph{true} subfactor.}) for $a\geq 2$ until it stabilizes at $M_c$. The entries of this tuple are the numbers of pages between the drops, i.e.\  the width of the steps in the staircase of objects $(E^a_{-c,c})_{c\geq 2}$. It follows that the sum over the entries of $\cp M$ is the number of pages it takes for $E^2_{-c,c}$ until it reaches $M_c$. In particular, a module is \emph{reflexively pure} if and only if $\cp M = (0)$.

The codegree of purity appears in Examples \ref{PurityFiltration} and \ref{PurityFiltration:A3}. In Example~\ref{CodegreeOfPurity} the codegree of purity is compared with two other classical homological invariants: \\
Recall, the \textbf{projective dimension} of a module $M$ is defined to be the length $d$ of the shortest projective resolution $0\xleftarrow{}M\xleftarrow{}P_0\xleftarrow{} \cdots \xleftarrow{} P_d\xleftarrow{} 0$. \textbf{{\sc Auslander}'s degree of torsion-freeness} of a module $M$ is defined following \cite[Def.~on~p.~2 \& Def.~2.15(b)]{AB} to be the smallest \emph{nonnegative} integer $i$, such that $\Ext^{i+1}(\mathrm{A}(M),D)\neq 0$, otherwise $\infty$, where $\mathrm{A}(M)$ is the so-called \textbf{{\sc Auslander} dual} of $M$ (see also \cite[Def.~5]{QEB}, \cite[Def.~19]{CQR05}). To construct $\mathrm{A}(M)$ take a projective presentation $0\xleftarrow{} M \xleftarrow{} P_0 \xleftarrow{d_1} P_1$ of $M$ and set
\[
  \mathrm{A}(M):=\coker(P_0^* \xrightarrow{d_1^*} P_1^*),
\]
where $d_1^* := \Hom(d_1, D)$ (cf.~\cite[p.~1 \& Def.~2.5]{AB}). Like the syzygies modules, it is proved in  \cite[Prop.~2.6(b)]{AB} that $\mathrm{A}(M)$ is uniquely determined by $M$ up to \textbf{projective equivalence} (see also \cite{Q} and \cite[Thm.~2]{PQ00}). In particular, the degree of torsion-freeness is well-defined. The fundamental exact sequence \cite[(0.1) \& Prop.~2.6(a)]{AB}
\[
  0 \xrightarrow{} \Ext^1_D(\mathrm{A}(M), -) \xrightarrow{}
  M \otimes_D - \xrightarrow{} \Hom_{D}( M^*, - )
  \xrightarrow{} \Ext^2_{D}(\mathrm{A}(M),-) \xrightarrow{} 0,
\]
applied to $D$, characterizes torsion-freeness and reflexivity of the module $M$ (see also \cite[Exercise~IV.7.3]{HS}, \cite[Thm.~6]{CQR05}). For a characterization of projectivity using the degree of torsion-freeness see \cite[Thm.~7]{CQR05}.

The codegree of purity can be defined for quasi-coherent sheaves of modules replacing $D$ by the structure sheaf $\mathcal{O}_X$ or by the dualizing sheaf\footnote{It may even be defined for objects in an abelian category with a dualizing object.} if it exists. It is important to note that the codegree of purity of a coherent sheaf $\mathcal{F}$ of $\mathcal{O}_X$-modules on a projective scheme $X=\Proj(S)$ may differ from the codegree of purity of a graded $S$-module $M$ used to represent $\mathcal{F}=\widetilde{M}=\Proj{M}$. This is mainly due to the fact that $\mathcal{F}=\widetilde{M}$ vanishes for {\sc Artin}ian modules $M$.

There are several obvious ways how one can refine the codegree of purity to get sharper invariants. The codegree of purity is an example of what can be called a \textbf{spectral invariant}.

\subsection{\texorpdfstring{The $\Tor$-$\Ext$ spectral sequence and the filtration of $\Ext$}{The Tor-Ext spectral sequence and the filtration of Ext}}


\begin{coro}[The $\Tor-\Ext$ spectral sequence]\label{TorExt}
  Let $M$ and $N$ be left $D$-modules. Then there exists a second quadrant cohomological spectral sequence with
  \[
    E_2^{pq}=\Tor_{-p}(\Ext^q(M,D),N) \Longrightarrow \Ext^{p+q}(M,N).
  \]
  In particular, there exists a descending filtration of $\Ext^{p+q}(M,N)$ with $\gr^p \Ext^{p+q}(M,N)$ naturally isomorphic to a subfactor of $\Tor_{-p}(\Ext^q(M,D),N)$, $p\leq 0$
\end{coro}


The special case $p+q=0$ recovers the filtration of $\Hom(M,N)$ mentioned in Example (d) of the Introduction via the natural isomorphism $\Hom(M,N) \cong \Ext^0 (M,N)$.

\bigskip
For \textbf{holonomic} modules $M$ over the {\sc Weyl} $k$-algebra $D:=A_n$ the special case formula
\[
  \Hom(M,N) \cong \Tor_n(\Ext^n(M,D),N)
\]
(cf.~\cite[Chap.~2, Thm.~7.15]{Bjo}) was used by {\sc H.~Tsai} and {\sc U.~Walther} in the case when also $N$ is holonomic to compute the finite dimensional $k$-vector space of homomorphisms \cite{TW}.

The induced filtration on $\Ext^1(M,N)$ can be used to attach a numerical invariant to each extension of $M$ with submodule $N$. This gives another example of a \textbf{spectral invariant}.

\subsubsection{Using the {\sc Grothendieck} bicomplex}\label{CGB}

Corollary~\ref{TorExt} is a consequence of Theorem~\ref{Gr2} for $F:=-\otimes_D N$ and $G:=\Hom_D(-,D)$ since $F\circ G$ coincides with $\Hom_D(-,N)$ on projectives. See Example~\ref{TorExt:Grothendieck}.

\subsubsection{Using the bicomplex $\Hom(P^M,P^N)$}\label{HPP}
The \textbf{bifunctoriality} of $\Hom$ leads to the following cohomological bicomplex
\[
  B := \Hom(P^M,P^N) \cong \Hom(P^M,D)\otimes P^N,
\]
where $P^L$ denotes a projective resolution of the module $L$. It is an easy exercise (cf.~\cite[Chap.~2, §4.14]{Bjo}) to show that starting from $r=2$ the first and second spectral sequence of $B$ coincide with those of the {\sc Grothendieck} bicomplex associated to $M$,  $F:=-\otimes_D N$ and $G:=\Hom_D(-,D)$. Both bicomplexes are constructive as only projective resolutions enter their definitions. The bicomplex $B$ has the computational advantage of avoiding the rather expensive {\sc Cartan-Eilenberg} resolution used to define the {\sc Grothendieck} bicomplex. See Example~\ref{TorExt:Bifunctor}. Compare the output of the command {\tt homalgRingStatistics} in Example~\ref{TorExt:Bifunctor} with corresponding output in Example~\ref{TorExt:Grothendieck}.

Since the first spectral sequence of the bicomplex $B:=\Hom(P^M,P^N)$ collapses, a small part of it is often used to compute $\Hom(M,N)$ over a \emph{commutative} ring $D$, as then all arrows of $B$ are again morphisms of $D$-modules. See \cite[p.~104]{GP} and \cite[Subsection~6.2.3]{BR}.

If the ring $D$ is \emph{not} commutative, then the above bicomplex and the {\sc Grothendieck} bicomplex in the previous subsection fail to be $D$-bicomplexes (unless when $M$ or $N$ is a $D$-bimodule). The bicomplexes are even in a lot of applications of interest \emph{not} of finite type over their natural domain of definition. In certain situations there nevertheless exist \emph{quasi-isomorphic} subfactor (bi)complexes which can be used to perform effective computations. In \cite{TW}, cited above, and in the pioneering work \cite{OT} {\sc Kashiwara}'s so-called $V$-filtration is used to extract such subfactors.

\newpage
\appendix

\section{The triangulation algorithm}

\begin{defn}[{Computable ring \cite[Subsection~1.2]{BR}}]\label{compdef}
  A left and right noetherian ring is called \textbf{computable} if there exists an algorithm which solves one sided inhomogeneous linear systems $XA=B$ and $AX=B$, where $A$ and $B$ are matrices with entries in $D$. The word ``solves'' means: The algorithm can decide if a solution exists, and, if solvable, is able to compute a particular solution as well as a finite generating set of solutions of the corresponding homogeneous system.
\end{defn}

From now on the ring $D$ is assumed computable. Let $M$ be a finitely generated left $D$-module. Then $M$ is finitely presented,  i.e. there exists a matrix $\tM \in D^{p\times q}$, viewed as a morphism $\tM:D^{1\times p} \to D^{1\times q}$, such that $\coker \tM \cong M$. $\tM$ is called a \textbf{matrix of relations} or a \textbf{presentation matrix} for $M$. It forms the beginning of a free resolution
\[
  0 \xleftarrow{} M \xleftarrow{} D^{1\times q} \xleftarrow{d_1 = \tM} D^{1\times p} \xleftarrow{d_2} D^{1\times p_2} \xleftarrow{d_3} \cdots.
\]
$d_i$ is called the $i$-th syzygies matrix of $M$ and $K_i:=\coker d_{i+1}$ the $i$-th syzygies module. $K_i$ is uniquely determined by $M$ up to \textbf{projective equivalence}.

Now suppose that $M$ is endowed with an $m$-filtration $F=(F_p M)$. We will sketch an algorithm that, starting from a presentation matrix $\tM\in D^{p \times q}$ for $M$ and presentation matrices $\tM_p$ for the graded parts $M_p:=\gr_p M$, computes another \textbf{upper triangular} presentation matrix $\tM_F$ of the form\footnote{Note that choosing a generating system of $M$ adapted to the filtration $F$ is not enough to produce a triangular presentation matrix, as changing the set of generators only corresponds to column operations on $\tM$.}
\[
  \tM_F = \left(
    \begin{matrix}
      \tM_{p_{m-1}} & * & \cdots & \cdots & * \\
      & \tM_{p_{m-2}} & * & \cdots & * \\
      & & \ddots & \ddots & \vdots \\
      & & & \tM_{p_1} & * \\
      & & & & \tM_{p_0}
    \end{matrix}
  \right) \in D^{p'\times q'}
\]
and an isomorphism $\coker \tM_F \xrightarrow{\cong} \coker \tM$ given by a matrix $\mathtt{T}\in D^{q' \times q}$:

Let $(\psi_p)$ be an ascending $m$-filtration system computing $F$ (cf.~Def.~\ref{system}). To start the induction take $p$ to be the highest degree $p_{m-1}$ in the filtration and set $\ttF_p \tM := \tM$. Since
\[
  \mu_p:=\psi_p:M_p=\coker\tM_p \to \coker \ttF_p\tM
\]
is a generalized isomorphism, its unique generalized inverse exists and is an epimorphism (cf.~Cor.~\ref{geninv}), which we denote by $\pi_p:F_p M \to M_p$. (Note: $\coker \ttF_p\tM=F_p M= M$ for $p=p_{m-1}$.) Since $D$ is computable we are able to determine (a matrix of) an injective morphism $\iota_p$ mapping onto the kernel of $\pi_p$. The source of $\iota_p$ is a module isomorphic to $F_{p-1} M$, which we also denote by $F_{p-1} M$. No confusion can occur since we will only refer to the latter. All maps in the exact-rows diagram
\[
  \xymatrix{
    0 & M_p \ar@{=}[d] \ar[l] & P_0 \ar[d]^{\eta_0} \ar[l]_{\nu} & K_1 \ar[d]^{\eta} \ar[l]_{\tM_p} & 0 \ar[l] \\
    0 & M_p \ar[l]& F_p M \ar[l]_{\pi_p} & F_{p-1}M \ar[l]_{\iota_p} & 0 \ar[l]
  }
\]
are computable, where $P_0$ is a free $D$-module and $K_1$ is the $1$-st syzygies module of $\tM_p$: $\eta_0$ is computable since $P_0$ is free and $\eta$ is computable since $\iota_p$ is injective (see~\cite[Subsection~3.1]{BR}). This yields the short exact sequence
\[
  0 \xrightarrow{} K_1 \xrightarrow{\kappa:=\left(\begin{matrix}\tM_p & \eta \end{matrix}\right)} P_0 \oplus F_{p-1} M \xrightarrow{\rho:=\left(\begin{matrix}-\eta_0 \\ \iota_p\end{matrix}\right)} F_p M \xrightarrow{} 0.
\]
Hence, the cokernel of $\kappa:=\left(\begin{matrix}\tM_p & \eta \end{matrix}\right)$ is isomorphic to $F_p M$ which therefore admits a presentation matrix of the form
\[
  \tM^p_F = \left(\begin{matrix} \tM_p & \eta \\ 0 & \ttF_{p-1} \tM \end{matrix}\right),
\]
where $\ttF_{p-1}\tM$ is a presentation matrix of $F_{p-1} M$ (for more details see~\cite[Subsection~7.1]{BB}). If $\chi: P_0 \oplus F_{p-1} M \xrightarrow{} \coker\kappa=\coker \tM^p_F$ denotes the natural epimorphism and $\rho:=\left(\begin{matrix}-\eta_0 \\ \iota_p\end{matrix}\right)$, then the matrix $\mathtt{T}^p$ of the morphism $T^p := \rho \circ\chi^{-1}$ is an isomorphism between $\coker\tM^p_F$ and $\coker \ttF_p\tM$. By the induction hypothesis we have
\[
  \widetilde{\tM}^{p+1}_F
  :=
  \left(\begin{array}{c|c} \mathrm{stable}_p & \eta_p \\ \hline 0 & \ttF_p \tM \end{array}\right)
  =
  \left( \begin{array}{c|c|c}
     \mathrm{stable}_{p+1} & * & * \\ \hline
     0 & M_{p+1} &  * \\ \hline
     0 & 0 & \ttF_p \tM
   \end{array} \right)
   =
   \left(\begin{array}{c|c} \mathrm{stable}_{p+1} & *\  * \\ \hline 0 & \tM^{p+1}_F \end{array}\right)
\]
with $\coker \widetilde{\tM}^{p+1}_F \cong \coker \tM$. (Since $p$ was decreased by one the old $\ttF_{p-1} \tM$ is now addressed as $\ttF_p \tM$, etc.). Before proceeding inductively on the submatrix $\ttF_p \tM$ of $\widetilde{\tM}^{p+1}_F$ take the quotient
\[
  \mu_p:=(\iota_{p_{m-1}}\circ\cdots\circ \iota_{p+1})^{-1}\circ\psi_p:M_p=\coker\tM_p \to \coker \ttF_p \tM,
\]
which is like $\mu_{p+1}$ again a generalized isomorphism. Note that matrix $\mathtt{T}^p$  of the morphism $T^p:=\rho \circ\chi^{-1}$ providing the isomorphism between $\coker\tM^p_F$ and $\coker \ttF_p\tM$ now has to be multiplied from the right to the submatrix $\eta_p$ of $\widetilde{\tM}^{p+1}_F$ which lies above $\ttF_p \tM$. This completes the induction. The algorithm terminates with $\tM_F := \widetilde{\tM}^{p_0}_F$ and $\mathtt{T}$ is the composition of all the successive column operations on $\tM$. \qed

\bigskip
The above algorithm is implemented in {\tt homalg}  package \cite{homalg-package} under the name {\tt Isomor\-phism\-OfFiltration}. It takes an $m$-filtration system $(\psi_p)$ of $M=\coker\tM$ as its input and returns an isomorphism $\tau:\coker \tM_F \to \coker \tM$ with a triangular presentation matrix $\tM_F$, as described above. {\tt IsomorphismOfFiltration} will be extensively used in the examples in Appendix~\ref{Examples}.

\section{Examples with {\sf GAP}'s {\tt homalg}}\label{Examples}

The packages {\tt homalg}, {\tt IO\_ForHomalg}, and {\tt RingsForHomalg} are assumed loaded:

\medskip
\noindent
{\color{blue}\verb+gap>+}{\color{OrangeRed}\verb+ LoadPackage( "RingsForHomalg" );+}
\begin{verbatim}
true
\end{verbatim}

For details see the {\tt homalg} project \cite{homalg-project}.


\begin{exmp}[{\tt LeftPresentation}]\label{LeftPresentation}
Define a left module $W$ over the polynomial ring $D:=\Q[x,y,z]$. Also define its right mirror $Y$.

\medskip
\noindent
{\small
{\color{blue}\verb+gap>+}{\color{OrangeRed}\verb+ Qxyz := HomalgFieldOfRationalsInDefaultCAS( ) * "x,y,z";;+} \\
{\color{blue}\verb+gap>+}{\color{OrangeRed}\verb+ wmat := HomalgMatrix( "[ \+
\begin{verbatim}
x*y,  y*z,    z,        0,         0,    \
x^3*z,x^2*z^2,0,        x*z^2,     -z^2, \
x^4,  x^3*z,  0,        x^2*z,     -x*z, \
0,    0,      x*y,      -y^2,      x^2-1,\
0,    0,      x^2*z,    -x*y*z,    y*z,  \
0,    0,      x^2*y-x^2,-x*y^2+x*y,y^2-y \
]", 6, 5, Qxyz );
\end{verbatim}
}
\begin{verbatim}
<A homalg external 6 by 5 matrix>
\end{verbatim}
{\color{blue}\verb+gap>+}{\color{OrangeRed}\verb+ W := LeftPresentation( wmat );+}
\begin{verbatim}
<A left module presented by 6 relations for 5 generators>
\end{verbatim}
{\color{blue}\verb+gap>+}{\color{OrangeRed}\verb+ Y := Hom( Qxyz, W );+}
\begin{verbatim}
<A right module on 5 generators satisfying 6 relations>
\end{verbatim}
}
\end{exmp}


\begin{exmp}[Homological {\tt GrothendieckSpectralSequence}]\label{ExtExt}
Example~\ref{LeftPresentation} continued. Compute the double-$\Ext$ spectral sequence for $F:=\Hom(-,Y)$, $G:=\Hom(-,D)$, and the $D$-module $W$. This is an example for Subsection~\ref{HGB}.

\medskip
\noindent
{\small
{\color{blue}\verb+gap>+}{\color{OrangeRed}\verb+ F := InsertObjectInMultiFunctor( Functor_Hom, 2, Y, "TensorY" );+}
\begin{verbatim}
<The functor TensorY>
\end{verbatim}
{\color{blue}\verb+gap>+}{\color{OrangeRed}\verb+ G := LeftDualizingFunctor( Qxyz );;+}\\
{\color{blue}\verb+gap>+}{\color{OrangeRed}\verb+ II_E := GrothendieckSpectralSequence( F, G, W );+}
\begin{verbatim}
<A stable homological spectral sequence with sheets at levels [ 0 .. 4 ]
each consisting of left modules at bidegrees [ -3 .. 0 ]x[ 0 .. 3 ]>
\end{verbatim}
{\color{blue}\verb+gap>+}{\color{OrangeRed}\verb+ Display( II_E );+}
\begin{verbatim}
The associated transposed spectral sequence:

a homological spectral sequence at bidegrees
[ [ 0 .. 3 ], [ -3 .. 0 ] ]
---------
Level 0:

 * * * *
 * * * *
 . * * *
 . . * *
---------
Level 1:

 * * * *
 . . . .
 . . . .
 . . . .
---------
Level 2:

 s s s s
 . . . .
 . . . .
 . . . .

Now the spectral sequence of the bicomplex:

a homological spectral sequence at bidegrees
[ [ -3 .. 0 ], [ 0 .. 3 ] ]
---------
Level 0:

 * * * *
 * * * *
 . * * *
 . . * *
---------
Level 1:

 * * * *
 * * * *
 . * * *
 . . . *
---------
Level 2:

 * * s s
 * * * *
 . * * *
 . . . *
---------
Level 3:

 * s s s
 * s s s
 . . s *
 . . . *
---------
Level 4:

 s s s s
 . s s s
 . . s s
 . . . s
\end{verbatim}
{\color{blue}\verb+gap>+}{\color{OrangeRed}\verb+ filt := FiltrationBySpectralSequence( II_E, 0 );+}
\begin{verbatim}
<An ascending filtration with degrees [ -3 .. 0 ] and graded parts:
   0:	<A non-zero left module presented by 33 relations for 23 generators>
  -1:	<A non-zero left module presented by 37 relations for 22 generators>
  -2:	<A non-zero left module presented by 20 relations for 8 generators>
  -3:	<A non-zero left module presented by 29 relations for 4 generators>
of
<A non-zero left module presented by 111 relations for 37 generators>>
\end{verbatim}
{\color{blue}\verb+gap>+}{\color{OrangeRed}\verb+ ByASmallerPresentation( filt );+}
\begin{verbatim}
<An ascending filtration with degrees [ -3 .. 0 ] and graded parts:
   0:	<A non-zero left module presented by 25 relations for 16 generators>
  -1:	<A non-zero left module presented by 30 relations for 14 generators>
  -2:	<A non-zero left module presented by 18 relations for 7 generators>
  -3:	<A non-zero left module presented by 12 relations for 4 generators>
of
<A non-zero left module presented by 48 relations for 20 generators>>
\end{verbatim}
{\color{blue}\verb+gap>+}{\color{OrangeRed}\verb+ m := IsomorphismOfFiltration( filt );+}
\begin{verbatim}
<An isomorphism of left modules>
\end{verbatim}
}
\end{exmp}


\begin{exmp}[{\tt PurityFiltration}]\label{PurityFiltration}
Example~\ref{LeftPresentation} continued. This is an example for Subsections~\ref{bidual}~and~\ref{codegree}.

\medskip
\noindent
{\small
{\color{blue}\verb+gap>+}{\color{OrangeRed}\verb+ filt := PurityFiltration( W );+}
\begin{verbatim}
<The ascending purity filtration with degrees [ -3 .. 0 ] and graded parts:
   0:	<A codegree-[ 1, 1 ]-pure rank 2 left module presented by
         3 relations for 4 generators>
  -1:	<A codegree-1-pure codim 1 left module presented by
         4 relations for 3 generators>
  -2:	<A cyclic reflexively pure codim 2 left module presented by
         2 relations for a cyclic generator>
  -3:	<A cyclic reflexively pure codim 3 left module presented by
         3 relations for a cyclic generator>
of
<A non-pure rank 2 left module presented by 6 relations for 5 generators>>
\end{verbatim}
{\color{blue}\verb+gap>+}{\color{OrangeRed}\verb+ W;+}
\begin{verbatim}
<A non-pure rank 2 left module presented by 6 relations for 5 generators>
\end{verbatim}
{\color{blue}\verb+gap>+}{\color{OrangeRed}\verb+ m := IsomorphismOfFiltration( filt );+}
\begin{verbatim}
<An isomorphism of left modules>
\end{verbatim}
{\color{blue}\verb+gap>+}{\color{OrangeRed}\verb+ IsIdenticalObj( Range( m ), W );+}
\begin{verbatim}
true
\end{verbatim}
{\color{blue}\verb+gap>+}{\color{OrangeRed}\verb+ Source( m );+}
\begin{verbatim}
<A left module presented by 12 relations for 9 generators (locked)>
\end{verbatim}
{\color{blue}\verb+gap>+}{\color{OrangeRed}\verb+ Display( last );+}
\begin{verbatim}
0,  0,   x, -y,0,1, 0,    0,  0,
x*y,-y*z,-z,0, 0,0, 0,    0,  0,
x^2,-x*z,0, -z,1,0, 0,    0,  0,
0,  0,   0, 0, y,-z,0,    0,  0,
0,  0,   0, 0, x,0, -z,   0,  1,
0,  0,   0, 0, 0,x, -y,   -1, 0,
0,  0,   0, 0, 0,-y,x^2-1,0,  0,
0,  0,   0, 0, 0,0, 0,    z,  0,
0,  0,   0, 0, 0,0, 0,    y-1,0,
0,  0,   0, 0, 0,0, 0,    0,  z,
0,  0,   0, 0, 0,0, 0,    0,  y,
0,  0,   0, 0, 0,0, 0,    0,  x 
Cokernel of the map

Q[x,y,z]^(1x12) --> Q[x,y,z]^(1x9),

currently represented by the above matrix
\end{verbatim}
{\color{blue}\verb+gap>+}{\color{OrangeRed}\verb+ Display( filt );+}
\begin{verbatim}
Degree 0:

0,  0,   x, -y,
x*y,-y*z,-z,0, 
x^2,-x*z,0, -z 
Cokernel of the map

Q[x,y,z]^(1x3) --> Q[x,y,z]^(1x4),

currently represented by the above matrix
----------
Degree -1:

y,-z,0,   
x,0, -z,  
0,x, -y,  
0,-y,x^2-1
Cokernel of the map

Q[x,y,z]^(1x4) --> Q[x,y,z]^(1x3),

currently represented by the above matrix
----------
Degree -2:

Q[x,y,z]/< z, y-1 >
----------
Degree -3:

Q[x,y,z]/< z, y, x >
\end{verbatim}
{\color{blue}\verb+gap>+}{\color{OrangeRed}\verb+ Display( m );+}
\begin{verbatim}
1,   0,    0,  0,   0, 
0,   -1,   0,  0,   0, 
0,   0,    -1, 0,   0, 
0,   0,    0,  -1,  0, 
-x^2,-x*z, 0,  -z,  0, 
0,   0,    x,  -y,  0, 
0,   0,    0,  0,   -1,
0,   0,    x^2,-x*y,y, 
x^3, x^2*z,0,  x*z, -z 

the map is currently represented by the above 9 x 5 matrix
\end{verbatim}
}
\end{exmp}


\begin{exmp}[{\tt PurityFiltration}, \emph{non}commutative]\label{PurityFiltration:A3}
This is a \emph{non}commutative example for Subsections~\ref{bidual}~and~\ref{codegree}. Let $A_3 := \Q[x,y,z]\langle D_x, D_y, D_z \rangle$ be the $3$-dimensional {\sc Weyl} algebra.

\medskip
\noindent
{\small
{\color{blue}\verb+gap>+}{\color{OrangeRed}\verb+ A3 := RingOfDerivations( Qxyz, "Dx,Dy,Dz" );;+}\\
{\color{blue}\verb+gap>+}{\color{OrangeRed}\verb+ nmat := HomalgMatrix( "[ \+
\begin{verbatim}
3*Dy*Dz-Dz^2+Dx+3*Dy-Dz,           3*Dy*Dz-Dz^2,     \
Dx*Dz+Dz^2+Dz,                     Dx*Dz+Dz^2,       \
Dx*Dy,                             0,                \
Dz^2-Dx+Dz,                        3*Dx*Dy+Dz^2,     \
Dx^2,                              0,                \
-Dz^2+Dx-Dz,                       3*Dx^2-Dz^2,      \
Dz^3-Dx*Dz+Dz^2,                   Dz^3,             \
2*x*Dz^2-2*x*Dx+2*x*Dz+3*Dx+3*Dz+3,2*x*Dz^2+3*Dx+3*Dz\
]", 8, 2, A3 );
\end{verbatim}
}
\begin{verbatim}
<A homalg external 8 by 2 matrix>
\end{verbatim}
{\color{blue}\verb+gap>+}{\color{OrangeRed}\verb+ N := LeftPresentation( nmat );+}
\begin{verbatim}
<A left module presented by 8 relations for 2 generators>
\end{verbatim}
{\color{blue}\verb+gap>+}{\color{OrangeRed}\verb+ filt := PurityFiltration( N );+}
\begin{verbatim}
<The ascending purity filtration with degrees [ -3 .. 0 ] and graded parts:
   0:	<A zero left module>
  -1:	<A cyclic reflexively pure codim 1 left module presented by 
         1 relation for a cyclic generator>
  -2:	<A cyclic reflexively pure codim 2 left module presented by 
         2 relations for a cyclic generator>
  -3:	<A cyclic reflexively pure codim 3 left module presented by 
         3 relations for a cyclic generator>
of
<A non-pure codim 1 left module presented by 8 relations for 2 generators>>
\end{verbatim}
{\color{blue}\verb+gap>+}{\color{OrangeRed}\verb+ II_E := SpectralSequence( filt );+}
\begin{verbatim}
<A stable homological spectral sequence with sheets at levels [ 0 .. 2 ]
each consisting of left modules at bidegrees [ -3 .. 0 ]x[ 0 .. 3 ]>
\end{verbatim}
{\color{blue}\verb+gap>+}{\color{OrangeRed}\verb+ Display( II_E );+}
\begin{verbatim}
The associated transposed spectral sequence:

a homological spectral sequence at bidegrees
[ [ 0 .. 3 ], [ -3 .. 0 ] ]
---------
Level 0:

 * * * *
 . * * *
 . . * *
 . . . *
---------
Level 1:

 * * * *
 . . . .
 . . . .
 . . . .
---------
Level 2:

 s . . .
 . . . .
 . . . .
 . . . .

Now the spectral sequence of the bicomplex:

a homological spectral sequence at bidegrees
[ [ -3 .. 0 ], [ 0 .. 3 ] ]
---------
Level 0:

 * * * *
 . * * *
 . . * *
 . . . *
---------
Level 1:

 * * * *
 . * * *
 . . * *
 . . . .
---------
Level 2:

 s . . .
 . s . .
 . . s .
 . . . .
\end{verbatim}
{\color{blue}\verb+gap>+}{\color{OrangeRed}\verb+ m := IsomorphismOfFiltration( filt );+}
\begin{verbatim}
<An isomorphism of left modules>
\end{verbatim}
{\color{blue}\verb+gap>+}{\color{OrangeRed}\verb+ IsIdenticalObj( Range( m ), N );+}
\begin{verbatim}
true
\end{verbatim}
{\color{blue}\verb+gap>+}{\color{OrangeRed}\verb+ Source( m );+}
\begin{verbatim}
<A left module presented by 6 relations for 3 generators (locked)>
\end{verbatim}
{\color{blue}\verb+gap>+}{\color{OrangeRed}\verb+ Display( last );+}
\begin{verbatim}
Dx,-1/3,-2/9*x,
0, Dy,  -1/3,  
0, Dx,  1,     
0, 0,   Dz,    
0, 0,   Dy,    
0, 0,   Dx     
Cokernel of the map

R^(1x6) --> R^(1x3), ( for R := Q[x,y,z]<Dx,Dy,Dz> )

currently represented by the above matrix
\end{verbatim}
{\color{blue}\verb+gap>+}{\color{OrangeRed}\verb+ Display( filt );+}
\begin{verbatim}
Degree 0:

0
----------
Degree -1:

Q[x,y,z]<Dx,Dy,Dz>/< Dx > 
----------
Degree -2:

Q[x,y,z]<Dx,Dy,Dz>/< Dy, Dx >
----------
Degree -3:

Q[x,y,z]<Dx,Dy,Dz>/< Dz, Dy, Dx >
\end{verbatim}
{\color{blue}\verb+gap>+}{\color{OrangeRed}\verb+ Display( m );+}
\begin{verbatim}
1,                1,     
-3*Dz-3,          -3*Dz, 
-3*Dz^2+3*Dx-3*Dz,-3*Dz^2

the map is currently represented by the above 3 x 2 matrix
\end{verbatim}
}
\end{exmp}


\begin{exmp}[Cohomological {\tt GrothendieckSpectralSequence}]\label{TorExt:Grothendieck}
Example~\ref{LeftPresentation} continued. Compute the $\Tor$-$\Ext$ spectral sequence for the triple $F:=-\otimes W$, $G:=\Hom(-,D)$, and the $D$-module $W$. This is an example for Subsection~\ref{CGB}.

\medskip
\noindent
{\small
{\color{blue}\verb+gap>+}{\color{OrangeRed}\verb+ F := InsertObjectInMultiFunctor( Functor_TensorProduct, 2, W, "TensorW" );+}
\begin{verbatim}
<The functor TensorW>
\end{verbatim}
{\color{blue}\verb+gap>+}{\color{OrangeRed}\verb+ G := LeftDualizingFunctor( Qxyz );;+}\\
{\color{blue}\verb+gap>+}{\color{OrangeRed}\verb+ II_E := GrothendieckSpectralSequence( F, G, W );+}
\begin{verbatim}
<A stable cohomological spectral sequence with sheets at levels [ 0 .. 4 ]
each consisting of left modules at bidegrees [ -3 .. 0 ]x[ 0 .. 3 ]>
\end{verbatim}
{\color{blue}\verb+gap>+}{\color{OrangeRed}\verb+ homalgRingStatistics(Qxyz);+}
\begin{verbatim}
rec( BasisOfRowModule := 110, BasisOfColumnModule := 16, 
  BasisOfRowsCoeff := 50, BasisOfColumnsCoeff := 60, DecideZeroRows := 241, 
  DecideZeroColumns := 31, DecideZeroRowsEffectively := 51, 
  DecideZeroColumnsEffectively := 63, SyzygiesGeneratorsOfRows := 184, 
  SyzygiesGeneratorsOfColumns := 63 )
\end{verbatim}
{\color{blue}\verb+gap>+}{\color{OrangeRed}\verb+ Display( II_E );+}
\begin{verbatim}
The associated transposed spectral sequence:

a cohomological spectral sequence at bidegrees
[ [ 0 .. 3 ], [ -3 .. 0 ] ]
---------
Level 0:

 * * * *
 * * * *
 . * * *
 . . * *
---------
Level 1:

 * * * *
 . . . .
 . . . .
 . . . .
---------
Level 2:

 s s s s
 . . . .
 . . . .
 . . . .

Now the spectral sequence of the bicomplex:

a cohomological spectral sequence at bidegrees
[ [ -3 .. 0 ], [ 0 .. 3 ] ]
---------
Level 0:

 * * * *
 * * * *
 . * * *
 . . * *
---------
Level 1:

 * * * *
 * * * *
 . * * *
 . . . *
---------
Level 2:

 * * s s
 * * * *
 . * * *
 . . . *
---------
Level 3:

 * s s s
 . s s s
 . . s *
 . . . s
---------
Level 4:

 s s s s
 . s s s
 . . s s
 . . . s
\end{verbatim}
{\color{blue}\verb+gap>+}{\color{OrangeRed}\verb+ filt := FiltrationBySpectralSequence( II_E, 0 );+}
\begin{verbatim}
<A descending filtration with degrees [ -3 .. 0 ] and graded parts:
  -3:	<A non-zero cyclic left module presented by
         3 relations for a cyclic generator>
  -2:	<A non-zero left module presented by 17 relations for 6 generators>
  -1:	<A non-zero left module presented by 19 relations for 9 generators>
   0:	<A non-zero left module presented by 13 relations for 10 generators>
of
<A left module presented by 66 relations for 41 generators>>
\end{verbatim}
{\color{blue}\verb+gap>+}{\color{OrangeRed}\verb+ ByASmallerPresentation( filt );+}
\begin{verbatim}
<A descending filtration with degrees [ -3 .. 0 ] and graded parts:
  -3:	<A non-zero cyclic left module presented by
         3 relations for a cyclic generator>
  -2:	<A non-zero left module presented by 12 relations for 4 generators>
  -1:	<A non-zero left module presented by 18 relations for 8 generators>
   0:	<A non-zero left module presented by 11 relations for 10 generators>
of
<A left module presented by 21 relations for 12 generators>>
\end{verbatim}
{\color{blue}\verb+gap>+}{\color{OrangeRed}\verb+ m := IsomorphismOfFiltration( filt );+}
\begin{verbatim}
<An isomorphism of left modules>
\end{verbatim}
}
\end{exmp}


\begin{exmp}[$\Tor$-$\Ext$ spectral sequence]\label{TorExt:Bifunctor}
Here we compute the $\Tor$-$\Ext$ spectral sequence of the bicomplex $B :=\Hom(P^W,D)\otimes P^W$. This is an example for Subsection~\ref{HPP}.

\medskip
\noindent
{\small
{\color{blue}\verb+gap>+}{\color{OrangeRed}\verb+ P := Resolution( W );+}
\begin{verbatim}
<A right acyclic complex containing 3 morphisms of left modules at degrees 
[ 0 .. 3 ]>
\end{verbatim}
{\color{blue}\verb+gap>+}{\color{OrangeRed}\verb+ GP := Hom( P );+}
\begin{verbatim}
<A cocomplex containing 3 morphisms of right modules at degrees [ 0 .. 3 ]>
\end{verbatim}
{\color{blue}\verb+gap>+}{\color{OrangeRed}\verb+ FGP := GP * P;+}
\begin{verbatim}
<A cocomplex containing 3 morphisms of left complexes at degrees [ 0 .. 3 ]>
\end{verbatim}
{\color{blue}\verb+gap>+}{\color{OrangeRed}\verb+ BC := HomalgBicomplex( FGP );+}
\begin{verbatim}
<A bicocomplex containing left modules at bidegrees [ 0 .. 3 ]x[ -3 .. 0 ]>
\end{verbatim}
{\color{blue}\verb+gap>+}{\color{OrangeRed}\verb+ p_degrees := ObjectDegreesOfBicomplex( BC )[1];+}
\begin{verbatim}
[ 0 .. 3 ]
\end{verbatim}
{\color{blue}\verb+gap>+}{\color{OrangeRed}\verb+ II_E := SecondSpectralSequenceWithFiltration( BC, p_degrees );+}
\begin{verbatim}
<A stable cohomological spectral sequence with sheets at levels [ 0 .. 4 ]
each consisting of left modules at bidegrees [ -3 .. 0 ]x[ 0 .. 3 ]>
\end{verbatim}
{\color{blue}\verb+gap>+}{\color{OrangeRed}\verb+ homalgRingStatistics(Qxyz);+}
\begin{verbatim}
rec( BasisOfRowModule := 109, BasisOfColumnModule := 1, 
  BasisOfRowsCoeff := 48, BasisOfColumnsCoeff := 0, DecideZeroRows := 190, 
  DecideZeroColumns := 1, DecideZeroRowsEffectively := 49, 
  DecideZeroColumnsEffectively := 0, SyzygiesGeneratorsOfRows := 166, 
  SyzygiesGeneratorsOfColumns := 2 )
\end{verbatim}
{\color{blue}\verb+gap>+}{\color{OrangeRed}\verb+ Display( II_E );+}
\begin{verbatim}
The associated transposed spectral sequence:

a cohomological spectral sequence at bidegrees
[ [ 0 .. 3 ], [ -3 .. 0 ] ]
---------
Level 0:

 * * * *
 * * * *
 * * * *
 * * * *
---------
Level 1:

 * * * *
 . . . .
 . . . .
 . . . .
---------
Level 2:

 s s s s
 . . . .
 . . . .
 . . . .

Now the spectral sequence of the bicomplex:

a cohomological spectral sequence at bidegrees
[ [ -3 .. 0 ], [ 0 .. 3 ] ]
---------
Level 0:

 * * * *
 * * * *
 * * * *
 * * * *
---------
Level 1:

 * * * *
 * * * *
 * * * *
 * * * *
---------
Level 2:

 * * s s
 * * * *
 . * * *
 . . . *
---------
Level 3:

 * s s s
 . s s s
 . . s *
 . . . s
---------
Level 4:

 s s s s
 . s s s
 . . s s
 . . . s
\end{verbatim}
{\color{blue}\verb+gap>+}{\color{OrangeRed}\verb+ filt := FiltrationBySpectralSequence( II_E, 0 );+}
\begin{verbatim}
<A descending filtration with degrees [ -3 .. 0 ] and graded parts:
  -3:	<A non-zero cyclic left module presented by 
         3 relations for a cyclic generator>
  -2:	<A non-zero left module presented by 17 relations for 7 generators>
  -1:	<A non-zero left module presented by 25 relations for 12 generators>
   0:	<A non-zero left module presented by 13 relations for 10 generators>
of
<A left module presented by 38 relations for 24 generators>>
\end{verbatim}
{\color{blue}\verb+gap>+}{\color{OrangeRed}\verb+ ByASmallerPresentation( filt );+}
\begin{verbatim}
<A descending filtration with degrees [ -3 .. 0 ] and graded parts:
  -3:	<A non-zero cyclic left module presented by 
3 relations for a cyclic generator>
  -2:	<A non-zero left module presented by 12 relations for 4 generators>
  -1:	<A non-zero left module presented by 21 relations for 8 generators>
   0:	<A non-zero left module presented by 11 relations for 10 generators>
of
<A left module presented by 23 relations for 12 generators>>
\end{verbatim}
{\color{blue}\verb+gap>+}{\color{OrangeRed}\verb+ m := IsomorphismOfFiltration( filt );+}
\begin{verbatim}
<An isomorphism of left modules>
\end{verbatim}
}
\end{exmp}


\begin{exmp}[{\tt CodegreeOfPurity}]\label{CodegreeOfPurity}
For two torsion-free $D$-modules $V$ and $W$ of rank $2$ compute the three homological invariants
\begin{itemize}
  \item projective dimension,
  \item {\sc Auslander}'s degree of torsion-freeness, and
  \item codegree of purity
\end{itemize}
mentioned in Subsection~\ref{codegree} are computed. The codegree of purity is able to distinguish the two modules.

\medskip
\noindent
{\small
{\color{blue}\verb+gap>+}{\color{OrangeRed}\verb+ vmat := HomalgMatrix( "[ \+
\begin{verbatim}
0,  0,  x,-z, \
x*z,z^2,y,0,  \
x^2,x*z,0,y   \
]", 3, 4, Qxyz );
\end{verbatim}
}
\begin{verbatim}
<A homalg external 3 by 4 matrix>
\end{verbatim}
{\color{blue}\verb+gap>+}{\color{OrangeRed}\verb+ V := LeftPresentation( vmat );+}
\begin{verbatim}
<A non-zero left module presented by 3 relations for 4 generators>
\end{verbatim}
{\color{blue}\verb+gap>+}{\color{OrangeRed}\verb+ wmat := HomalgMatrix( "[ \+
\begin{verbatim}
0,  0,  x,-y, \
x*y,y*z,z,0,  \
x^2,x*z,0,z   \
]", 3, 4, Qxyz );
\end{verbatim}
}
\begin{verbatim}
<A homalg external 3 by 4 matrix>
\end{verbatim}
{\color{blue}\verb+gap>+}{\color{OrangeRed}\verb+ W := LeftPresentation( wmat );+}
\begin{verbatim}
<A non-zero left module presented by 3 relations for 4 generators>
\end{verbatim}
{\color{blue}\verb+gap>+}{\color{OrangeRed}\verb+ Rank( V );+}
\begin{verbatim}
2
\end{verbatim}
{\color{blue}\verb+gap>+}{\color{OrangeRed}\verb+ Rank( W );+}
\begin{verbatim}
2
\end{verbatim}
{\color{blue}\verb+gap>+}{\color{OrangeRed}\verb+ ProjectiveDimension( V );+}
\begin{verbatim}
2
\end{verbatim}
{\color{blue}\verb+gap>+}{\color{OrangeRed}\verb+ ProjectiveDimension( W );+}
\begin{verbatim}
2
\end{verbatim}
{\color{blue}\verb+gap>+}{\color{OrangeRed}\verb+ DegreeOfTorsionFreeness( V );+}
\begin{verbatim}
1
\end{verbatim}
{\color{blue}\verb+gap>+}{\color{OrangeRed}\verb+ DegreeOfTorsionFreeness( W );+}
\begin{verbatim}
1
\end{verbatim}
{\color{blue}\verb+gap>+}{\color{OrangeRed}\verb+ CodegreeOfPurity( V );+}
\begin{verbatim}
[ 2 ]
\end{verbatim}
{\color{blue}\verb+gap>+}{\color{OrangeRed}\verb+ CodegreeOfPurity( W );+}
\begin{verbatim}
[ 1, 1 ]
\end{verbatim}
{\color{blue}\verb+gap>+}{\color{OrangeRed}\verb+ filtV := PurityFiltration( V );+}
\begin{verbatim}
<The ascending purity filtration with degrees [ -2 .. 0 ] and graded parts:
   0:	<A codegree-[ 2 ]-pure rank 2 left module presented by
         3 relations for  4 generators>
  -1:	<A zero left module>
  -2:	<A zero left module>
of
<A codegree-[ 2 ]-pure rank 2 left module presented by
3 relations for 4 generators>>
\end{verbatim}
{\color{blue}\verb+gap>+}{\color{OrangeRed}\verb+ filtW := PurityFiltration( W );+}
\begin{verbatim}
<The ascending purity filtration with degrees [ -2 .. 0 ] and graded parts:
   0:	<A codegree-[ 1, 1 ]-pure rank 2 left module presented by 
         3 relations for 4 generators>
  -1:	<A zero left module>
  -2:	<A zero left module>
of
<A codegree-[ 1, 1 ]-pure rank 2 left module presented by
3 relations for 4 generators>>
\end{verbatim}
{\color{blue}\verb+gap>+}{\color{OrangeRed}\verb+ II_EV := SpectralSequence( filtV );+}
\begin{verbatim}
<A stable homological spectral sequence with sheets at levels [ 0 .. 4 ]
each consisting of left modules at bidegrees [ -3 .. 0 ]x[ 0 .. 2 ]>
\end{verbatim}
{\color{blue}\verb+gap>+}{\color{OrangeRed}\verb+ Display( II_EV );+}
\begin{verbatim}
The associated transposed spectral sequence:

a homological spectral sequence at bidegrees
[ [ 0 .. 2 ], [ -3 .. 0 ] ]
---------
Level 0:

 * * *
 * * *
 * * *
 . * *
---------
Level 1:

 * * *
 . . .
 . . .
 . . .
---------
Level 2:

 s . .
 . . .
 . . .
 . . .

Now the spectral sequence of the bicomplex:

a homological spectral sequence at bidegrees
[ [ -3 .. 0 ], [ 0 .. 2 ] ]
---------
Level 0:

 * * * *
 * * * *
 . * * *
---------
Level 1:

 * * * *
 * * * *
 . . * *
---------
Level 2:

 * . . .
 * . . .
 . . * *
---------
Level 3:

 * . . .
 . . . .
 . . . *
---------
Level 4:

 . . . .
 . . . .
 . . . s
\end{verbatim}
{\color{blue}\verb+gap>+}{\color{OrangeRed}\verb+ II_EW := SpectralSequence( filtW );+}
\begin{verbatim}
<A stable homological spectral sequence with sheets at levels [ 0 .. 4 ]
each consisting of left modules at bidegrees [ -3 .. 0 ]x[ 0 .. 2 ]>
\end{verbatim}
{\color{blue}\verb+gap>+}{\color{OrangeRed}\verb+ Display( II_EW );+}
\begin{verbatim}
The associated transposed spectral sequence:

a homological spectral sequence at bidegrees
[ [ 0 .. 2 ], [ -3 .. 0 ] ]
---------
Level 0:

 * * *
 * * *
 . * *
 . . *
---------
Level 1:

 * * *
 . . .
 . . .
 . . .
---------
Level 2:

 s . .
 . . .
 . . .
 . . .

Now the spectral sequence of the bicomplex:

a homological spectral sequence at bidegrees
[ [ -3 .. 0 ], [ 0 .. 2 ] ]
---------
Level 0:

 * * * *
 . * * *
 . . * *
---------
Level 1:

 * * * *
 . * * *
 . . . *
---------
Level 2:

 * . . .
 . * . .
 . . . *
---------
Level 3:

 * . . .
 . . . .
 . . . *
---------
Level 4:

 . . . .
 . . . .
 . . . s
\end{verbatim}
}
\end{exmp}


\bigskip
An alternative title for this work could have been "Squeezing spectral sequences".


\def\cprime{$'$} \def\cprime{$'$}
\providecommand{\bysame}{\leavevmode\hbox to3em{\hrulefill}\thinspace}
\providecommand{\MR}{\relax\ifhmode\unskip\space\fi MR }
\providecommand{\MRhref}[2]{%
  \href{http://www.ams.org/mathscinet-getitem?mr=#1}{#2}
}
\providecommand{\href}[2]{#2}



\end{document}